\documentclass[reqno]{amsart}
\usepackage{bm}
\usepackage{cite}
\usepackage{color}
\usepackage{hyperref}

\newtheorem{thm}{Theorem}
\newtheorem{lem}{Lemma}

\newtheorem{coro}{Corollary}
\newtheorem{defi}{Definition}
\newtheorem{rem}{Remark}
\newtheorem{prob}{Problem}

\def\BC{\mathbb C}
\def\BN{\mathbb N}
\def\BR{\mathbb R}
\def\cA{\mathcal A}
\def\cD{\mathcal D}
\def\cP{\mathcal P}
\def\cQ{\mathcal Q}
\def\cS{\mathcal S}
\def\Ga{\Gamma}
\def\Si{\Sigma}
\def\Om{\Omega}
\def\al{\alpha}
\def\be{\beta}
\def\ga{\gamma}
\def\de{\delta}

\def\ve{\varepsilon}
\def\ze{\zeta}
\def\te{\theta}
\def\ka{\kappa}
\def\la{\lambda}
\def\vp{\varphi}

\def\f{\frac}
\def\nb{\nabla}
\def\ov{\overline}
\def\un{\underline}
\def\pa{\partial}
\def\wh{\widehat}
\def\wt{\widetilde}
\def\T{\mathrm T}
\def\rd{\mathrm d}
\def\diag{\mathrm{diag}}
\def\e{\mathrm e}
\def\ri{\mathrm i}
\def\rRe{\mathrm{Re}}
\def\rIm{\mathrm{Im}}

\numberwithin{equation}{section}
\allowdisplaybreaks[4]

\begin{document}
\title[Coupled time-fractional diffusion systems]{Initial-boundary value problems for coupled systems of time-fractional diffusion equations}

\author[Z. Li]{Zhiyuan LI}
\author[X. Huang]{Xinchi Huang}
\author[Y. Liu]{Yikan LIU}

\address{Zhiyuan Li\newline
School of Mathematics and Statistics, Ningbo University, 818 Fenghua Road, Ningbo, Zhejiang 315211, China.}
\email{lizhiyuan@nbu.edu.cn}
\address{Xinchi Huang \newline
Graduate School of Mathematical Sciences, the University of Tokyo, 3-8-1 Komaba, Meguro-ku, Tokyo 153-8914, Japan.}
\email{huangxc@ms.u-tokyo.ac.jp}
\address{Yikan Liu\newline
Research Center of Mathematics for Social Creativity, Research Institute for Electronic Science, Hokkaido University, N12W7, Kita-Ward, Sapporo 060-0812, Japan.}
\email{ykliu@es.hokudai.ac.jp}


\begin{abstract}
This article deals with the initial-boundary value problem for a moderately coupled system of time-fractional diffusion equations. Defining the mild solution, we establish fundamental unique existence, limited smoothing property and long-time asymptotic behavior of the solution, which mostly inherit those of a single equation. Owing to the coupling effect, we also obtain the uniqueness for an inverse problem on determining all the fractional orders by the single point observation of a single component of the solution.
\end{abstract}

\maketitle

\section{Introduction}\label{sect-intro}

For anomalous diffusion in heterogeneous media and viscoelastic body that conventional partial differential equations (PDEs) fail to describe, a considerable number of nonlocal models with fractional derivatives have been proposed. Especially, due to the capability of representing the time memory effect, time-fractional PDEs such as
\begin{equation}\label{eq-TFPDE}
(\pa_t^\al-\triangle)u=F
\end{equation}
with a Caputo derivative $\pa_t^\al$ of order $0<\al<2$ in time (which will be defined soon) have gathered popularity among mathematicians and multidisciplinary researchers. The past decade has witnessed a tremendous development in mathematical analysis related to \eqref{eq-TFPDE} and its various generalizations: fundamental theories such as the unique existence of solutions have been established (e.g.\! \cite{EK04,GLY15,J21,KRY20,SY11}), and corresponding numerical and inverse problems have also been investigated extensively 
(e.g.\! \cite{JLLZ15,JLZ13,LiLiuY19,LY19,LiuLiY19}). In particular, for time-fractional diffusion equations with $\al<1$ in \eqref{eq-TFPDE}, a linear theory equivalent to their parabolic prototypes has been well constructed, and their similarity and difference have been clarified.

Among the rapidly increasing literature on fractional-related topics, however, it turns out that almost all existing researches are restricted to single and linear time-fractional PDEs. As we know, for important applications in chemistry, biology and finance, coupled systems of PDEs (represented by reaction-diffusion systems) successfully model the interaction and evolution of two or more involved components under consideration. Therefore, in the context of systems with more sophisticated mechanism due to a richer structure associated to the corresponding couplings, it is desirable to create new models based on coupled systems of time-fractional PDEs or even general nonlocal reaction-diffusion systems (e.g. \cite{EI21,EIKMT21}). On the other hand, now that a linear theory for single time-fractional diffusion equations is accomplished, it seems natural to study its generalization to linear systems first and investigate them from both theoretical and practical aspects.

Keeping the above backgrounds in mind, let us set up the formulation in this article.

Let $T\in\BR_+:=(0,+\infty)$ be a constant and $\Om\subset\BR^d$ ($d\in\BN:=\{1,2,\dots\}$) be an open bounded domain with a smooth boundary $\pa\Om$ (for example, of $C^2$ class). For a constant $K\in\BN$, let $\al_1,\dots,\al_K$ be constants satisfying $1>\al_1\ge\cdots\ge\al_K>0$. 
In this article, we consider the following initial-boundary value problem for a coupled system of time-fractional diffusion equations
\begin{equation}\label{eq-ibvp-u0}
\left\{\begin{alignedat}{2}
& \begin{aligned}
\pa_t^{\al_k}u_k & =\mathrm{div}(\bm A_k(\bm x)\nb u_k)+\sum_{\ell=1}^K\bm b_{k\ell}(\bm x,t)\cdot\nb u_\ell\\
& \quad\,+\sum_{\ell=1}^K c_{k\ell}(\bm x,t)u_\ell+F_k(\bm x,t)
\end{aligned}
& \quad & \mbox{in }\Om\times(0,T),\\
& u_k=u_0^{(k)} & \quad & \mbox{in }\Om\times\{0\},\\
& u_k=0 & \quad & \mbox{on }\pa\Om\times(0,T),
\end{alignedat}\right.\quad k=1,\dots,K,
\end{equation}
where by $\pa_t^\al$ ($0<\al<1$) we denote the Caputo derivative (e.g.\! Podlubny \cite{P99}) which is usually defined as
\[
\pa_t^\al f(t):=\f1{\Ga(1-\al)}\int_0^t\f{f'(\tau)}{(t-\tau)^\al}\,\rd\tau,\quad t>0,\ f\in C^1[0,+\infty).
\]
There are several ways to define the Caputo derivative and the domain could be extended from $C^1$ to some fractional Sobolev space (e.g.\! \cite{GLY15}). Since our main concentration in this article is the coupled system, we do not discuss further on the definition of the fractional derivative. 
Here $\bm A_k=(a_{i j}^{(k)})_{1\le i,j\le d}\in C^1(\ov\Om;\BR_{\mathrm{sym}}^{d\times d})$ ($k=1,\dots,K$) are symmetric and strictly positive-definite matrices on $\ov\Om$\,. 
More precisely, there exists a constant $\ka>0$ such that
\[
\bm A_k(\bm x)=(\bm A_k(\bm x))^\T,\quad\bm A_k(\bm x)\bm\xi\cdot\bm\xi\ge\ka|\bm\xi|^2,\quad\forall\,\bm\xi\in\BR^d,\ \forall\,\bm x\in\ov\Om,\ \forall\,k=1,\dots,K,
\]
where $(\,\cdot\,)^\T$ denotes the transpose and $|\bm\xi|^2:=\bm\xi\cdot\bm\xi$. Later in Section \ref{sec-premain} we will provide further details of the assumptions on involved initial values, source terms and coefficients.

The governing equation in \eqref{eq-ibvp-u0} is a moderately coupled system of $K$ linear time-fractional diffusion equations, where the components are allowed to interact with each other up to the first derivative in space. The orders $\al_k$ of time derivatives for all components can be different from each other, and the coefficients of zeroth and first-order spatial derivatives can depend on both $\bm x$ and $t$. Nevertheless, the coefficients of principal parts are restricted to be $t$-independent, since later we shall rely on the existence of eigensystems of elliptic operators $\cA_k\psi=-\mathrm{div}(\bm A_k(\bm x)\nb\psi)$. In this sense, the homogeneous Dirichlet boundary condition in \eqref{eq-ibvp-u0} is not obligatory and can be replaced by homogeneous Neumann or Robin ones. Therefore, the formulation \eqref{eq-ibvp-u0} covers a rather wide range of problems, which seems not yet studied in literature to the best of our knowledge.

Indeed, it reveals that existing publications, though limited, emphasize more on the nonlinear counterpart of \eqref{eq-ibvp-u0}, whereas the elliptic parts can be reasonably simple. Recently, Gal and Warma \cite[Chapter 4]{GW20} considered systems of fractional kinetic equations with nonlinear terms taking the form of $F_k(\bm x,t,u_1,\dots,u_K)$ and discussed the existence of maximal mild and strong solutions. Very recently, Suzuki \cite{S21,S22} investigated local existence and non-existence for weakly coupled time-fractional reaction-diffusion systems with moderately and rapidly growing nonlinearities. In view of applications (e.g.\! \cite{LHW07,LHYJH20}), definitely one should consider nonlinear reaction-diffusion systems based on \eqref{eq-ibvp-u0} and develop corresponding theories. However, as was seen in the research of traditional reaction-diffusion systems, linear coupled systems play fundamental roles especially in the discussions of super/subsolution methods and traveling waves. Hence, as the starting point, the major target of this article is to construct the basic well-posedness results for \eqref{eq-ibvp-u0} parallel to the case of a single equation.

On the other hand, we are also interested in inverse problems associated with \eqref{eq-ibvp-u0}. Again in the framework of traditional parabolic systems, sometimes it is possible to identify multiple coefficients simultaneously from observation data of a single component due to the coupling effect. Thus, the natural curiosity is whether and to which extent such property is inherited by the fractional systems. On this direction, only Ren, Huang and Yamamoto \cite{RHY21} obtained the conditional stability for a coefficient inverse problem of \eqref{eq-ibvp-u0} in the special case of $K=2$, $\al_1=\al_2=1/2$ and $d=1$ by means of Carleman estimates. Hence, in this article we also keep an eye on the minor target of studying a small inverse problem for \eqref{eq-ibvp-u0} employing the coupling effect (see Problem \ref{prob-IP}).

For the forward problem, we prove the unique existence of a mild solution to \eqref{eq-ibvp-u0} along with its stability with respect to initial values $u_0^{(k)}$ and source terms $F_k$ in Theorem \ref{thm-well}. These results turn out to be parallel to those for single equations obtained in \cite{GLY15,SY11}, namely, the improvement of spatial regularity of solutions from that of $u_0^{(k)},F_k$ is at most $2$. 
The proof generalizes the idea in \cite{GLY15,LHY20} to regard the lower order terms as a part of source terms and construct a sequence by Picard iteration, whose limit is indeed the mild solution. Further, restricting the system \eqref{eq-ibvp-u0} to a weakly coupled one with $t$-independent coefficients, we investigate the long-time asymptotic behavior of the solution and obtain the sharp decay rate $t^{-\al_K}$ in Theorem \ref{thm-asymp}, where $\al_K$ stands for the lowest order of the fractional derivatives. Such a result coincides with that for a multi-term time-fractional diffusion equation proved in \cite{LLY15}. 
Meanwhile, we study a parameter inverse problem on determining all the orders $\al_1,\dots,\al_K$ by the observation of a single component $u_{k_0}$ at $\{\bm x_0\}\times(0,T)$. Using the strong maximum principle for coupled elliptic systems, in Theorem \ref{thm-ip} we show the uniqueness of this inverse problem provided that the system is not decoupled. This reflects the interaction between components as expected, which is only available for coupled systems.

The rest of this paper is organized as follows. In Section \ref{sec-premain}, we collect necessary preliminaries and give statements of the main results. Then Sections \ref{sec-well}--\ref{sec-alpha} are devoted to the proofs of Theorems \ref{thm-well},\ref{thm-asymp} and \ref{thm-ip}, respectively. Some concluding remarks are provided in Section \ref{sec-rem}, and finally the proofs of some technical details are postponed to Appendix \ref{sec-app}.

\section{Preliminaries and main results}\label{sec-premain}

We start from fixing some general settings and notations. Let $L^2(\Om)$ be the possibly complex-valued square-integrable function space and $H_0^1(\Om)$, $H^2(\Om)$, $H^{-1}(\Om)$, $W^{1,\infty}(\Om)$ etc.\! be the usual Sobolev spaces (e.g.\! \cite{Adams}). The inner products of $L^2(\Om)$ and $\BC^K$ ($K\in\BN$) are defined by
\begin{gather*}
(f,g):=\int_\Om f(\bm x)\,\ov{g(\bm x)}\,\rd\bm x,\quad f,g\in L^2(\Om),\\
\bm\xi\cdot\bm\ze:=\sum_{k=1}^K\xi_k\ov{\ze_k},\quad\bm\xi=(\xi_1,\dots,\xi_K)^\T,\bm\ze=(\ze_1,\dots,\ze_K)^\T\in\BC^K,
\end{gather*}
respectively, which induce the respective norms
\[
\|f\|_{L^2(\Om)}:=(f,f)^{1/2},\quad|\bm\xi|:=(\bm\xi\cdot\bm\xi)^{1/2}.
\]
With slight abuse of notation, we also denote the length of a multi-index $\bm j=(j_1,\dots,j_K)\in(\BN\cup\{0\})^K$ by $|\bm j|$, i.e., $|\bm j|=\sum_{k=1}^K j_k$. Given a Banach space $X$, the norm of the product space $X^K$ ($K\in\BN$) is defined as
\[
\|\bm f\|_{X^K}:=\sum_{k=1}^K\|f_k\|_X,\quad\bm f=(f_1,\dots,f_K)^\T\in X^K.
\]
Throughout this article, we abbreviate $\|\bm\psi\|_{X^K}$ as $\|\bm\psi\|_X$ for the sake of conciseness. Similarly, the inner product of $(L^2(\Om))^K$ is abbreviated as $(\bm f,\bm g)$: 
\[
(\bm f,\bm g)=\int_\Om\bm f\cdot\bm g\,\rd\bm x=\sum_{k=1}^K(f_k,g_k),\quad\bm f=(f_1,\dots,f_K)^\T,\bm g=(g_1,\dots,g_K)^\T\in(L^2(\Om))^K.
\]

For the coefficients $\bm b_{k\ell}$ and $c_{k\ell}$ in \eqref{eq-ibvp-u0}, we assume
\begin{equation}\label{eq-reg-bc}
\bm b_{k\ell}\in (L^\infty(0,T;W^{j,\infty}(\Om)))^d,\quad c_{k\ell}\in L^\infty(0,T;W^{j,\infty}(\Om)),\quad k,\ell=1,\dots,K
\end{equation}
with $j=0$ or $j=1$. Later we will see that the choice of $j$ influences the regularity of solutions.

The governing equations in \eqref{eq-ibvp-u0} look lengthy and we shall introduce some notations for later convenience. In the sequel, we denote $\bm u:=(u_1,\dots,u_K)^\T$ and introduce second-order elliptic operators $\cA_k$ along with first-order differential operators $\cP_k$ ($k=1,\dots,K$) as
\begin{gather*}
\cA_k:\cD(\cA_k):=H^2(\Om)\cap H_0^1(\Om)\longrightarrow L^2(\Om),\quad\psi\longmapsto-\mathrm{div}(\bm A_k(\bm x)\nb\psi),\\
\cP_k:(H_0^1(\Om))^K\longrightarrow L^2(\Om),\quad\bm\psi=(\psi_1,\dots,\psi_K)^\T \longmapsto\sum_{\ell=1}^K(\bm b_{k\ell}(\bm x,t)\cdot\nb\psi_\ell+c_{k\ell}(\bm x,t)\psi_\ell).
\end{gather*}
Then we can rewrite the governing equations in \eqref{eq-ibvp-u0} as
\[
(\pa_t^{\al_k}+\cA_k)u_k=\cP_k\bm u+F_k\quad\mbox{in }\Om\times(0,T),\ k=1,\dots,K.
\]
Further introducing
\begin{gather*}
\bm\al:=(\al_1,\dots,\al_K)^\T,\quad\bm F(\bm x,t):=(F_1(\bm x,t),\dots,F_K(\bm x,t))^\T,\quad\bm u_0(\bm x):=(u_0^{(1)}(\bm x),\dots,u_0^{(K)}(\bm x))^\T,\\
\pa_t^{\bm\al}\bm u:=(\pa_t^{\al_1}u_1,\dots,\pa_t^{\al_K}u_K)^\T,\quad\cA\bm u:=(\cA_1u_1,\dots,\cA_K u_K)^\T,\quad \cP\bm u:=(\cP_1\bm u,\dots,\cP_K\bm u)^\T,
\end{gather*}
we can collect the above $K$ equations to represent \eqref{eq-ibvp-u0} in a vector form as
\begin{equation}\label{eq-ibvp-u1}
\begin{cases}
(\pa_t^{\bm\al}+\cA)\bm u=\cP\bm u+\bm F & \mbox{in }\Om\times(0,T),\\
\bm u=\bm u_0 & \mbox{in }\Om\times\{0\},\\
\bm u=\bm 0 & \mbox{on }\pa\Om\times(0,T).
\end{cases}
\end{equation}
In the sequel, we abbreviate $\bm u(t)=\bm u(\,\cdot\,,t)$ as a vector-valued $\bm x$-dependent function.

Next, we introduce the eigensystem $\{(\la_n^{(k)},\vp_n^{(k)})\}_{n\in\BN}$ of the operator $\cA_k$ ($k=1,\dots,K$). More precisely, the sequence $\{\la_n^{(k)}\}\subset\BR_+$ satisfies
\[
0<\la_1^{(k)}<\la_2^{(k)}\le\cdots,\quad\lim_{n\to\infty}\la_n^{(k)}=+\infty,\quad k=1,\dots,K
\]
and $\{\vp_n^{(k)}\}\subset\cD(\cA_k)$ forms a complete orthonormal basis of $L^2(\Om)$ and satisfies
\[
\begin{cases}
\cA_k\vp_n^{(k)}=\la_n^{(k)}\vp_n^{(k)} & \mbox{in }\Om,\\
\vp_n^{(k)}=0 & \mbox{on }\pa\Om,
\end{cases}\quad n\in\BN.
\]
Then for $\ga>0$ and $k=1,\dots,K$, one can define the fractional power $\cA_k^\ga$ as (e.g.\! Pazy \cite{P83})
\[
\cD(\cA_k^\ga):=\left\{\psi\in L^2(\Om)\left|\,\sum_{n=1}^\infty\left|(\la_n^{(k)})^\ga(\psi,\vp_n^{(k)})\right|^2<\infty\right.\right\},\quad\cA_k^\ga\psi:=\sum_{n=1}^\infty(\la_n^{(k)})^\ga(\psi,\vp_n^{(k)})\vp_n^{(k)},
\]
which is equipped with the norm
\[
\|\psi\|_{\cD(\cA_k^\ga)}:=\left(\sum_{n=1}^\infty\left|(\la_n^{(k)})^\ga(\psi,\vp_n^{(k)})\right|^2\right)^{1/2}.
\]
For $-1\le\ga<0$ and $k=1,\dots,K$, we define $\cD(\cA_k^\ga)$ as the dual space of $\cD(\cA_k^{-\ga})$ similarly with the norm
\[
\|\psi\|_{\cD(\cA_k^\ga)}:=\left(\sum_{n=1}^\infty\left|(\la_n^{(k)})^\ga\langle\psi,\vp_n^{(k)}\rangle\right|^2\right)^{1/2},
\]
where $\langle\,\cdot\,,\,\cdot\,\rangle$ denotes the pairing between $\cD(\cA_k^\ga)$ and $\cD(\cA_k^{-\ga})$. Then $\cD(\cA_k^\ga)$ is well-defined for all $\ga\ge-1$ and we have
\[
\|\psi\|_{\cD(\cA_k^\ga)}=\|\cA_k^\ga\psi\|_{L^2(\Om)},\quad\psi\in\cD(\cA_k^\ga).
\]
We know that $\cD(\cA_k^\ga)$ is a Hilbert space and satisfies $\cD(\cA_k^\ga)\subset H^{2\ga}(\Om)$ for $\ga>0$. Especially, there hold $\cD(\cA_k^{1/2})=H_0^1(\Om)$, $\cD(\cA_k^{-1/2})=H^{-1}(\Om)$ and the norm equivalence $\|\cdot\|_{\cD(\cA_k^\ga)}\sim\|\cdot\|_{H^{2\ga}(\Om)}$. 
Similarly as before, for $\bm\psi=(\psi_1,\dots,\psi_K)^\T$ and $\ga\ge-1$, we denote
\[
\cA^\ga\bm\psi:=(\cA_1^\ga\psi_1,\dots,\cA_K^\ga\psi_K)^\T
\]
and define $\cD(\cA^\ga)$ as well as its norm $\|\cdot\|_{\cD(\cA^\ga)}$ correspondingly.

Now we invoke the frequently used Mittag-Leffler function (e.g.\! \cite{P99})
\[
E_{\al,\be}(z):=\sum_{m=0}^\infty\f{z^m}{\Ga(\al m+\be)},\quad\al\in\BR_+,\ \be\in\BR,\ z\in\BC,
\]
where $\Ga(\,\cdot\,)$ is the Gamma function. The following estimate of $E_{\al,\be}(z)$ is well-known.

\begin{lem}[{see \cite[Theorem 1.5]{P99}}]\label{lem-ML}
Let $0<\al<2$ and $\be>0$. 
Then there exists a constant $C_0>0$ depending only on $\al,\be$ such that
\[
0<|E_{\al,\be}(z)|\le C_0(1+z)^{(1-\be)/\al}\exp(z^{1/\al}),\quad z>0.
\]
\end{lem}

Employing the Mittag-Leffler functions, we further define the resolvent operator $\cS_k(z):L^2(\Om)\longrightarrow L^2(\Om)$ as well as its termwise differentiation $\cS_k'(z):L^2(\Om)\longrightarrow L^2(\Om)$ for $z\in\BC\setminus\{0\}$ and $k=1,\dots,K$ by
\[
\begin{aligned}
\cS_k(z)\psi & :=\sum_{n=1}^\infty E_{\al_k,1}(-\la_n^{(k)} z^{\al_k})(\psi,\vp_n^{(k)})\vp_n^{(k)},\\
\cS_k'(z)\psi & :=-z^{\al_k-1}\sum_{n=1}^\infty\la_n^{(k)}E_{\al_k,\al_k}(-\la_n^{(k)} z^{\al_k})(\psi,\vp_n^{(k)})\vp_n^{(k)},
\end{aligned}\quad\psi\in L^2(\Om).
\]
We recall the key properties of $\cS_k(z)$ and $\cS_k'(z)$.

\begin{lem}[see \cite{LHY20}]\label{lem-Sk}
For $k=1,\dots,K$ and $\psi\in L^2(\Om),$ the followings hold true.

{\rm(i)} For $0\le\ga\le1,$ there exists a constant $C_1>0$ depending only on $\Om,\bm\al,\cA$ such that
\begin{align}
\|\cA_k^\ga\cS_k(z)\psi\|_{L^2(\Om)} & \le C_1\|\psi\|_{L^2(\Om)}|z|^{-\al_k\ga},\nonumber\\
\|\cA_k^{\ga-1}\cS_k'(z)\psi\|_{L^2(\Om)} & \le C_1\|\psi\|_{L^2(\Om)}|z|^{\al_k(1-\ga)-1}\label{eq-Sk2}
\end{align}
for all $z\in\{z\in\BC\setminus\{0\}\mid|\arg z|<\pi/2\}$.

{\rm(ii)} There holds $\lim_{z\to0}\|\cS_k(z)\psi-\psi\|_{L^2(\Om)}=0$.
\end{lem}

In the same manner as before, for $\bm\psi=(\psi_1,\dots,\psi_K)^\T$ we denote
\[
\cS(z)\bm\psi:=(\cS_1(z)\psi_1,\dots,\cS_K(z)\psi_K)^\T.
\]

Let us close the preliminaries by considering the initial-boundary value problem
\begin{equation}\label{eq-ibvp-v}
\begin{cases}
(\pa_t^{\bm\al}+\cA)\bm v=\bm G & \mbox{in }\Om\times(0,T),\\
\bm v=\bm v_0 & \mbox{in }\Om\times\{0\},\\
\bm v=\bm0 & \mbox{on }\pa\Om\times(0,T)
\end{cases}
\end{equation}
for $K$ independent equations of $\bm v=(v_1,\dots,v_K)^\T$. The following lemma is a direct consequence of the well-posedness results for single equations.

\begin{lem}\label{lem-ibvp-v}
Let $\bm v_0\in(L^2(\Om))^K$ and $\bm G\in(L^p(0,T;L^2(\Om)))^K$ with $p\in[1,\infty]$.

{\rm(i)} If $\bm G\equiv\bm 0,$ then \eqref{eq-ibvp-v} admits a unique solution $\bm v\in L^{1/\ga}(0,T;\cD(\cA^\ga))$ $(0\le\ga\le1)$ which takes the form
\begin{equation}\label{eq-rep-v1}
\bm v(t)=\cS(t)\bm v_0,\quad0<t<T
\end{equation}
and satisfies
\begin{equation}\label{eq-IC-v}
\lim_{t\to0}\|\bm v(t)-\bm v_0\|_{L^2(\Om)}=0.
\end{equation}
Here we understand $1/\ga=\infty$ if $\ga=0$. Moreover$,$ there holds
\begin{equation}\label{eq-est-v1}
\|\bm v(t)\|_{\cD(\cA^\ga)}\le C_1\sum_{k=1}^K\|v_0^{(k)}\|_{L^2(\Om)}t^{-\al_k\ga},\quad0<t<T,\ 0\le\ga\le1,
\end{equation}
where $C_1$ is the constant in Lemma $\ref{lem-Sk}$.

{\rm(ii)} If $\bm v_0\equiv\bm 0,$ then \eqref{eq-ibvp-v} admits a unique solution $\bm v\in L^p(0,T;\cD(\cA^\ga))$ which takes the form
\begin{equation}\label{eq-rep-v2}
\bm v(t)=-\int_0^t\cA^{-1}\cS'(t-\tau)\bm G(\tau)\,\rd\tau,\quad0<t<T,
\end{equation}
where $\ga=1$ if $p=2$ and $0\le\ga<1$ if $p\ne2$. Moreover$,$ there exists a constant $C_2>0$ depending only on $\Om,\bm\al,\cA,T,\ga$ such that
\begin{equation}\label{eq-est-v2}
\|\bm v\|_{L^p(0,T;\cD(\cA^\ga))}\le C_2\|\bm G\|_{L^p(0,T;L^2(\Om))}.
\end{equation}
\end{lem}

The solution representations \eqref{eq-rep-v1} and \eqref{eq-rep-v2} as well as the convergence \eqref{eq-IC-v} are well-known in literature (see e.g.\! \cite{SY11}), and the estimate \eqref{eq-est-v1} follows immediately from Lemma \ref{lem-Sk}. As for the estimate \eqref{eq-est-v2}, we refer to \cite[Theorem 1.4]{GLY15} for the case of $p=2$ and \cite[Theorem 2.2]{LLY15} for the case of $p\ne2$. In particular, the constant $C_2$ in \eqref{eq-est-v2} may tend to $\infty$ as $\ga\to1$ if $p\ne2$, while keeps uniform with respect to $\ga$ if $p=2$.

Now we are well prepared to investigate the initial-boundary value problem \eqref{eq-ibvp-u0} or equivalently \eqref{eq-ibvp-u1}. Following the same line of \cite{GLY15,LHY20} and regarding the lower order term $\cP\bm u$ in \eqref{eq-ibvp-u1} as a part of the source term, we employ the representations \eqref{eq-rep-v1} and \eqref{eq-rep-v2} to conclude that the solution to \eqref{eq-ibvp-u1} should formally satisfy
\begin{equation}\label{eq-rep-u}
\bm u=\bm w+\cQ\bm u\quad\mbox{in }\Om\times(0,T),
\end{equation}
where
\begin{gather}
\bm w(t):=\cS(t)\bm u_0-\int_0^t\cA^{-1}\cS^\prime(t-\tau)\bm F(\tau)\,\rd\tau,\nonumber\\
\cQ\bm u(t):=-\int_0^t\cA^{-1}\cS'(t-\tau)(\cP\bm u)(\tau)\,\rd\tau.\label{eq-def-Q}
\end{gather}
This encourages us to propose the following definition of a solution to \eqref{eq-ibvp-u1}.

\begin{defi}[Mild solution]\label{def-mild}
Let $\bm u_0\in(L^2(\Om))^K,$ $\bm F\in(L^p(0,T;L^2(\Om)))^K$ with $p\in[1,\infty]$ and assume \eqref{eq-reg-bc} with $j=0$. We say that $\bm u$ is a mild solution to the initial-boundary value problem \eqref{eq-ibvp-u1} if it satisfies the integral equation \eqref{eq-rep-u}.
\end{defi}

Now we state the first main result in this article, which validates the well-posedness of 
the initial-boundary value problem \eqref{eq-ibvp-u1} defined above.

\begin{thm}\label{thm-well}
Under the same assumptions in Definition $\ref{def-mild}$, the followings hold true.

{\rm(i)} If $\bm F\equiv\bm 0,$ then there exists a unique mild solution $\bm u\in L^{1/\ga}(0,T;\cD(\cA^\ga))$ $(0\le\ga\le1)$ to \eqref{eq-ibvp-u1} satisfying
\begin{equation}\label{eq-IC-u}
\lim_{t\to0}\|\bm u(t)-\bm u_0\|_{L^2(\Om)}=0.
\end{equation}
Here we understand $1/\ga=\infty$ if $\ga=0$. Moreover$,$ there exist constants $C>0$ and $C_{T,\ga}>0$ depending only on $\Om,\bm\al,\cA,\cP$ and $\Om,\bm\al,\cA,\cP,T,\ga,$ respectively such that
\begin{gather}
\|\bm u(t)\|_{\cD(\cA^\ga)}\le C\exp(C\,t)\|\bm u_0\|_{L^2(\Om)}t^{-\al_1\ga},\quad0<t<T,\ 0\le\ga<1,\label{eq-est-u0}\\
\|\bm u\|_{L^{1/\ga}(0,T;\cD(\cA^\ga))}\le C_{T,\ga}\|\bm u_0\|_{L^2(\Om)},\quad0\le\ga\le1.\label{eq-est-u1}
\end{gather}
If we further assume \eqref{eq-reg-bc} with $j=1,$ then \eqref{eq-est-u0} also holds for $\ga=1$. Moreover$,$ the solution $\bm u:(0,T]\longrightarrow\cD(\cA^\ga)$ is analytic for $\ga\in[0,1)$.

{\rm(ii)} Let $\ga=1$ if $p=2$ and $0\le\ga<1$ if $p\ne2$. If $\bm u_0\equiv\bm 0,$ then \eqref{eq-ibvp-u1} admits a unique mild solution $\bm u\in L^p(0,T;\cD(\cA^\ga))$. Moreover$,$ there exists a constant $C_{T,\ga}>0$ depending only on $\Om,\bm\al,\cA,\cP,T,\ga$ such that
\begin{equation}\label{eq-est-u2}
\|\bm u\|_{L^p(0,T;\cD(\cA^\ga))}\le C_{T,\ga}\|\bm F\|_{L^p(0,T;L^2(\Om))}.
\end{equation}
\end{thm}

In the above theorem, the system \eqref{eq-ibvp-u1} is coupled up to first derivatives in space. In the sequel, we additionally assume
\begin{equation}\label{eq-weak}
\bm b_{k\ell}\equiv\bm 0,\quad c_{k\ell}(\bm x,t)=c_{k\ell}(\bm x),\quad k,\ell=1,\dots,K
\end{equation}
in \eqref{eq-ibvp-u1}, that is, we restrict \eqref{eq-ibvp-u1} to a weakly coupled system with $t$-independent zeroth order coefficients. For simplicity, by introducing a matrix-valued function $\bm C:=(c_{k\ell})_{1\le k,\ell\le K}$, it is obvious that $\cP\bm u=\bm C\bm u$ in \eqref{eq-ibvp-u1}.

Under the above assumption, we discuss the long-time asymptotic behavior of the solution to \eqref{eq-ibvp-u1}. The same topics were considered by many authors in the case of a single equation (see e.g.\! \cite{LLY15,SY11} and the references therein). On the basis of Theorem \ref{thm-well}, we establish the following theorem.

\begin{thm}\label{thm-asymp}
Let $T=+\infty,$ $\bm F=\bm 0$ and $\bm u_0\in(L^2(\Om))^K$ in \eqref{eq-ibvp-u1}. Assume \eqref{eq-weak} and that $\bm C=(c_{k\ell})_{1\le k,\ell\le K}\in(L^\infty(\Om))^{K\times K}$  is negative semi-definite in $\Om$. Then the unique solution to \eqref{eq-ibvp-u1} admits the asymptotic behavior
$$
\|\bm u(\,\cdot\,,t)\|_{H^2(\Om)}\le C\,t^{-\al_K}\|\bm u_0\|_{L^2(\Om)},\quad\forall\,t\ge t_0
$$
for arbitrarily fixed $t_0>0$. Here the constant $C>0$ is independent of $t,\bm u_0$ but depends on $t_0,K,\Om,\bm\al,\bm C$ and $\cA$. Moreover$,$ the decay rate $t^{-\al_K}$ is sharp provided that $u_0^{(K)}\not\equiv 0$ in $\Om$.
\end{thm}

The above theorem asserts that the decay rate of the solution $\bm u$ to \eqref{eq-ibvp-u1} is $t^{-\al_K}$ at best as $t\to+\infty$, where $\al_K$ is the lowest order of the time-fractional derivatives. This means that if $\|\bm u(\,\cdot\,,t)\|_{H^2(\Om)}=o(t^{-\al_K})$ as $t\to+\infty$, then $\bm u$ must vanish identically in $\Om\times\BR_+$, so that we call the decay rate $t^{-\al_K}$ sharp. This coincides with the result in the multi-term case considered by \cite{LLY15} in which the decay rate of the solution was shown to be dominated only by the lowest order of the fractional derivatives.

On the same direction, next we consider the following inverse problem.

\begin{prob}\label{prob-IP}
Let $\bm u$ satisfy \eqref{eq-ibvp-u1} and fix $T>0,$ $\bm x_0\in\Om,$ $k_0\in\{1,\dots,K\}$ arbitrarily. Determine the orders $\bm\al=(\al_1,\dots,\al_K)^\T$ of \eqref{eq-ibvp-u1} by the single point observation of the $k_0$-th component $u_{k_0}$ of $\bm u$ at $\{\bm x_0\}\times(0,T)$.
\end{prob}

The orders of the Caputo derivatives in time-fractional partial differential equations are related to important physical parameters describing e.g.\! the heterogeneity of media, whose determination is of great interest from both applied and mathematical aspects. Therefore, similar inverse problems have been considered intensively in the case of a single equation (see e.g.\! \cite{HNWY13,JK22,LHY20b,LIY15,LY15,LY22} and the survey \cite{LiLiuY19}). Here we attempt to extend the result from a single equation to a weakly coupled system. Since the components of the solution $\bm u$ interact each other, in Problem \ref{prob-IP} we ask the possibility of identifying all the orders $\bm\al$ by only observing a single component. The following theorem gives an affirmative answer to the uniqueness for Problem \ref{prob-IP}.

\begin{thm}\label{thm-ip}
Let $d=1,2,3,$ $\bm F=\bm 0$ and $\bm u_0\in(L^2(\Om))^K$ satisfy $u_0^{(k)}\ge0,\not\equiv 0$ in $\Om$ for $k=1,\dots,K$. Assume \eqref{eq-weak} and that $\bm C=(c_{k\ell})_{1\le k,\ell\le K}\in(L^\infty(\Om))^{K\times K}$ satisfies
\begin{gather}
c_{k\ell}\ge 0,\not\equiv 0\quad\mbox{on }\ov\Om,\ k,\ell=1,\dots,K,\ k\ne\ell,\label{eq-cond-C1}\\
\sum_{\ell=1}^K c_{k\ell}\le0\quad\mbox{on }\ov\Om,\ k=1,\dots,K.\label{eq-cond-C2}
\end{gather}
Further, let $\bm u$ and $\bm v$ be the solutions to \eqref{eq-ibvp-u1} with the fractional orders $\bm\al$ and $\bm\be,$ respectively. Then for arbitrarily fixed $T>0,$ $\bm x_0\in\Om$ and $k_0\in\{1,\dots,K\},$ we conclude that $u_{k_0}=v_{k_0}$ at $\{\bm x_0\}\times(0,T)$ implies $\bm\al=\bm\be$.
\end{thm}

In the above theorem, we additionally assume that the dimension $d=1,2,3$ since we need the Sobolev embedding $H^2(\Om)\subset C(\ov\Om)$ in the proof. Theorem $\ref{thm-ip}$ also holds true for $d>3$ provided that the coefficients in $\cA,\bm C$ and the initial value $\bm u_0$ are sufficiently smooth, but here we omit the details. 
Moreover, we interpret a function $f\in L^2(\Om)$ satisfying $f\ge 0,\not\equiv 0$ as $f\ge0$ in $\Om$ and the measure of $\{\bm x\in\Om\mid f(\bm x)>0\}$ is not zero.

\begin{rem}\label{rem-ip2}
{\rm The conditions \eqref{eq-cond-C1}--\eqref{eq-cond-C2} imposed on the matrix $\bm C$ are sufficient but not necessary for Theorem \ref{thm-ip}. Such conditions are understood as cooperativeness in a coupled system. Actually, one can find a similar condition
\begin{equation}\label{eq-cond-C1'}
c_{k\ell}\ge0\quad\mbox{on }\ov\Om,\ k,\ell=1,\dots,K,\ k\ne\ell
\end{equation}
and \eqref{eq-cond-C2} in many other papers (e.g.\! \cite{FM90,S89}) which are known as classical sufficient conditions for the maximum principle for weakly coupled elliptic systems. However, we need a slightly stronger condition \eqref{eq-cond-C1} than \eqref{eq-cond-C1'} in Theorem \ref{thm-ip} because we attempt to determine all the orders by only one component measurement of the solution, which requires that the system \eqref{eq-ibvp-u1} should not be decoupled (i.e.\! the decoupled case $c_{k\ell}\equiv 0$ for all $k\ne\ell$ should be excluded). On the other hand, it is readily seen from the proof that under the weaker assumption \eqref{eq-cond-C1'} (i.e. the decoupled case is allowed), one can obtain the uniqueness of determining the orders by observing all the components of $\bm u$. Moreover, as we state in the following corollary, it is possible to choose different observation point for each component, that is, we can observe the $k$-th component $u_k$ at $\bm x_0^{(k)}$ for each $k$.}
\end{rem}

\begin{coro}
Under the same assumptions on $d,\bm F,\bm u_0$ as those in Theorem $\ref{thm-ip},$ assume \eqref{eq-weak} and that $\bm C=(c_{k\ell})_{1\le k,\ell\le K}\in(L^\infty(\Om))^{K\times K}$ satisfies the conditions \eqref{eq-cond-C2} and \eqref{eq-cond-C1'}. Further$,$ let $\bm u$ and $\bm v$ be the solutions to \eqref{eq-ibvp-u1} with the fractional orders $\bm\al$ and $\bm\be,$ respectively. 
Then for arbitrarily fixed $T>0$ and $\bm x_0^{(1)},\dots, \bm x_0^{(K)}\in\Om,$ we conclude that $u_k(\bm x_0^{(k)},t)=v_k(\bm x_0^{(k)},t)$ for all $k=1,\dots,K$ and $t\in (0,T)$ implies $\bm\al=\bm\be$.
\end{coro}

\section{Proof of Theorem \ref{thm-well}}\label{sec-well}

In the integral equation \eqref{eq-rep-u}, it is readily seen that $\bm w$ relies only on $\bm u_0$ and $\bm F$, which is well analyzed in view of Lemma \ref{lem-ibvp-v}. On the other hand, the term $\cQ\bm u$ involves the solution $\bm u$ itself, which is more essential in the discussion of the unique existence of the mild solution. To this end, we should first investigate the operator $\cQ$.

\begin{lem}\label{lem-est-Q}
Let $j=0$ in the condition \eqref{eq-reg-bc} and $1/2\le\ga<1,$ $p\in[1,+\infty]$ and
\[
\cQ:L^p(0,T;\cD(\cA^\ga))\longrightarrow L^p(0,T;\cD(\cA^\ga))
\]
be the linear operator defined by \eqref{eq-def-Q}. Then for any $m\in\BN$ and a.e.\! $t\in(0,T),$ there holds
\begin{equation}\label{eq-est-Q}
\|\cQ^m\bm v(t)\|_{\cD(\cA^\ga)}\le M^m\left(\sum_{k=1}^K J^{\al_k(1-\ga)}\right)^m\|\bm v(t)\|_{\cD(\cA^\ga)}
\end{equation}
for $\bm v\in L^p(0,T;\cD(\cA^\ga)),$ where
\begin{align}
& M:=C_1L_0\max_{1\le k\le K}\Ga(\al_k(1-\ga))=C_1L_0\,\Ga(\al_K(1-\ga)),\label{eq-def-Mj}\\
& L_0:=\max\left\{1,\sum_{k,\ell=1}^K\left(\|\bm b_{k\ell}\|_{L^\infty(\Om\times(0,T))}+\|c_{k\ell}\|_{L^\infty(\Om\times(0,T))}\right)\right\}.\nonumber
\end{align}
Here $C_1$ is the constant in Lemma $\ref{lem-Sk}$.
\end{lem}

\begin{proof}
Let us show \eqref{eq-est-Q} by induction and first deal with the case of $m=1$. Noticing that $\cP$ is a first-order differential operator with $L^\infty$ coefficients, we utilize \eqref{eq-Sk2} in Lemma \ref{lem-Sk} to estimate for a.e.\! $t\in(0,T)$ that
\begin{align*}
\|\cQ\bm v(t)\|_{\cD(\cA^\ga)} & =\sum_{k=1}^K\left\|\int_0^t\cA_k^{-1}\cS_k'(t-\tau)(\cP_k\bm v)(\tau)\,\rd\tau\right\|_{\cD(\cA_k^\ga)}\\
& \le\sum_{k=1}^K\int_0^t\left\|\cA_k^{\ga-1}\cS_k'(t-\tau)\left(\cP_k\bm v\right)(\tau)\right\|_{L^2(\Om)}\rd\tau\\
& \le C_1\sum_{k=1}^K\int_0^t\left\|\left(\cP_k\bm v\right)(\tau)\right\|_{L^2(\Om)}(t-\tau)^{\al_k(1-\ga)-1}\,\rd\tau\\
& =C_1L_0\sum_{k=1}^K\int_0^t\|\bm v(\tau)\|_{\cD(\cA^\ga)}(t-\tau)^{\al_k(1-\ga)-1}\,\rd\tau.
\end{align*}
By the definition of the Riemann-Liouville fractional integral, we further derive
\begin{align}
\|\cQ\bm v(t)\|_{\cD(\cA^\ga)} & \le C_1L_0\sum_{k=1}^K\Ga(\al_k(1-\ga))J^{\al_k(1-\ga)}\left(\|\bm v(t)\|_{\cD(\cA^\ga)}\right)\nonumber\\
& \le C_1L_0\max_{1\le k\le K}\Ga(\al_k(1-\ga))\sum_{k=1}^K J^{\al_k(1-\ga)}\left(\|\bm v(t)\|_{\cD(\cA^\ga)}\right).\label{esti-Qu}
\end{align}
This verifies \eqref{eq-est-Q} for $m=1$ by the definition \eqref{eq-def-Mj} of $M$. For $m\ge 2$, assume that \eqref{eq-est-Q} holds for $\ell=1,\dots,m-1$. Then for a.e.\! $t\in(0,T)$, we employ the inductive assumption to estimate
\begin{align*}
\|\cQ^m\bm v(t)\|_{\cD(\cA^\ga)} & =\|\cQ^{m-1}(\cQ\bm v)(t)\|_{\cD(\cA^\ga)}\\
& \le M^{m-1}\left(\sum_{k=1}^K J^{\al_k(1-\ga)}\right)^{m-1}\left(\|\cQ\bm v(t)\|_{\cD(\cA^\ga)}\right)\\
& \le M^{m-1}\left(\sum_{k=1}^K J^{\al_k(1-\ga)}\right)^{m-1}\left(M\sum_{k=1}^K J^{\al_k(1-\ga)}\|\bm v(t)\|_{\cD(\cA^\ga)}\right)\\
& =M^m\left(\sum_{k=1}^K J^{\al_k(1-\ga)}\right)^m\|\bm v(t)\|_{\cD(\cA^\ga)}.
\end{align*}
Here the last inequality is due to the estimate \eqref{esti-Qu}. This indicates that \eqref{eq-est-Q} holds for any $m\in\BN$, which completes the proof.
\end{proof}

On the basis of \eqref{eq-rep-u}, we can construct a sequence $\{\bm u^{(m)}\}_{m=0}^\infty$ iteratively by
\begin{equation}\label{eq-iteration}
\bm u^{(m)}(t):=\begin{cases}
\bm 0, & m=0,\\
\bm w(t) + \cQ\bm u^{(m-1)}(t), & m=1,2,\dots.
\end{cases}
\end{equation}
Our strategy is to show the convergence of $\{\bm u^{(m)}\}$ in some function spaces, whose limit is indeed a mild solution of \eqref{eq-ibvp-u1}. For proving the time-analyticity of the solution, it is convenient to consider the above iteration \eqref{eq-iteration} in the complex plane. To this end, we first change variables and rewrite \eqref{eq-iteration} as
$$
\bm u^{(m)}(t)=\bm w(t)-t\int_0^1\cA^{-1}\cS'(\tau t)(\cP \bm u^{(m-1)}((1-\tau)t))\,\rd\tau.
$$
Moreover, we extend the variable $t$ from $(0,T)$ to the sector $S:=\{z\in\BC\setminus\{0\}\mid|\arg z|<\pi/2\}$ to obtain
\begin{align*}
\bm u^{(m)}(z) & =\bm w(z)-z\int_0^1\cA^{-1}\cS'(\tau z)(\cP \bm u^{(m-1)}((1-\tau)z))\,\rd\tau\\
& =:\bm w(z)+\cQ \bm u^{(m-1)}(z),\quad m=1,2,\dots.
\end{align*}

To proceed, we shall give some estimates for $\bm u^{(m)}$ in appropriate spaces.

\begin{lem}\label{lem-est-um}
Under the assumptions in Theorem $\ref{thm-well},$ define a sequence $\{\bm u^{(m)}\}_{m=0}^\infty$ by \eqref{eq-iteration}. Let $C_1$ and $C_2$ be the constants in Lemmas $\ref{lem-Sk}$ and $\ref{lem-ibvp-v},$ respectively, and $M$ be defined by \eqref{eq-def-Mj}.

{\rm(i)} If $\bm F\equiv\bm 0,$ then for $m=0,1,\dots$ and $1/2\le\ga<1,$ there holds
\begin{equation}\label{eq-est-um1}
\|(\bm u^{(m+1)}-\bm u^{(m)})(z)\|_{\cD(\cA^\ga)}\le C_3M^{m+1}\|\bm u_0\|_{L^2(\Om)}\sum_{|\bm j|=m}\f{m!}{j_1!\cdots j_K!}\f{|z|^{\bm\be\cdot\bm j-\al_1\ga}}{\Ga(\bm\be\cdot\bm j+1-\al_1\ga)}
\end{equation}
for $z\in S$ satisfying $|z|\le T,$ where $\be_k:=\al_k(1-\ga),$ $\bm\be=(\be_1,\dots,\be_K)^\T,$ $\bm j=(j_1,\dots,j_K)^\T\in(\BN\cup\{0\})^K$ and
\[
C_3:=C_1\max_{1\le k\le K} T^{(\al_1-\al_k)\ga}.
\]

{\rm(ii)} If $\bm u_0\equiv\bm 0,$ let $\ga=1$ for $p=2$ and $0\le\ga<1$ for $p\ne2,$ then for $m=0,1,\dots,$ there holds
\begin{equation}\label{eq-est-um2}
\|\bm u^{(m+1)}-\bm u^{(m)}\|_{L^p(0,T;\cD(\cA^\ga))}\le C_4(M K)^m\|\bm F\|_{L^p(0,T;L^2(\Om))}\f{\max\{1,T^{\ov\be m}\}}{\Ga(\un\be m+1)},
\end{equation}
where $\ov\be:=\max\{\be_1,\dots,\be_K\},$ $\un\be:=\min\{\be_1,\dots,\be_K\}$ and 
$$
C_4 := C_2\left(\min_{y\ge0}\Ga(1+y)\right)^{-1}.
$$
\end{lem}

\begin{proof}
(i) If $\bm F\equiv\bm 0$, then $\bm w(z)=\cS(z)\bm u_0$. Taking advantage of the estimate \eqref{eq-est-v1} in Lemma \ref{lem-ibvp-v}, we obtain
\begin{align*}
\|\bm w(z)\|_{\cD(\cA^\ga)} & \le C_1\sum_{k=1}^K\|u_0^{(k)}\|_{L^2(\Om)}|z|^{-\al_k\ga}\le C_1|z|^{-\al_1\ga}\sum_{k=1}^K T^{(\al_1-\al_k)\ga}\|u_0^{(k)}\|_{L^2(\Om)}\\
& =C_3|z|^{-\al_1\ga}\|\bm u_0\|_{L^2(\Om)},
\end{align*}
which is exactly \eqref{eq-est-um1} for $m=0$. 
Next, for any $m\in\BN$, using a method similar to estimating $\cQ\bm v(t)$ in Lemma \ref{lem-est-Q}, we can get
\begin{align*}
\|\cQ\bm v(z)\|_{\cD(\cA^\ga)} & =|z|\sum_{k=1}^K\left\|\int_0^1 \cA_k^{-1}\cS_k'(\tau z)(\cP_k\bm v)((1-\tau)z)\,\rd\tau\right\|_{\cD(\cA_k^\ga)}\\
& \le|z|\sum_{k=1}^K\int_0^1\left\|\cA_k^{\ga-1}\cS_k'(\tau z)\left(\cP_k\bm v\right)((1-\tau)z)\right\|_{L^2(\Om)}\rd\tau\\
& \le C_1\sum_{k=1}^K|z|^{\al_k(1-\ga)}\int_0^1\|(\cP_k\bm v)(\tau z)\|_{L^2(\Om)}(1-\tau)^{\al_k(1-\ga)-1}\,\rd\tau\\
& \le C_1L_0\sum_{k=1}^K\int_0^{|z|}\|\bm v(\tau z/|z|)\|_{\cD(\cA^\ga)}(|z|-\tau)^{\al_k(1-\ga)-1}\,\rd\tau,
\end{align*}
which, combined with the inductive assumption, implies that
\begin{align*}
& \quad\,\,\left\|(\bm u^{(m+1)}-\bm u^{(m)})(z)\right\|_{\cD(\cA^\ga)}\le\left\|\cQ(\bm u^{(m)}-\bm u^{(m-1)})(z)\right\|_{\cD(\cA^\ga)}\\
& \le C_1L_0\sum_{k=1}^K\int_0^{|z|}\left\|(\bm u^{(m)}-\bm u^{(m-1)})(\tau z/|z|)\right\|_{\cD(\cA^\ga)}(|z|-\tau)^{\al_k(1-\ga)-1}\,\rd\tau\\
& \le C_1L_0C_3M^m\|\bm u_0\|_{L^2(\Om)}\sum_{k=1}^K\int_0^{|z|}\sum_{|\bm j|=m-1}\f{(m-1)!}{j_1!\cdots j_K!}\f{\tau^{\bm\be\cdot\bm j-\al_1\ga}}{\Ga(\bm\be\cdot\bm j+1-\al_1\ga)}(|z|-\tau)^{\al_k(1-\ga)-1}\,\rd\tau.
\end{align*}
Moreover, reminding the definition of the Riemann-Liouville integral operator, we can rewrite the above terms as
\begin{align*}
& \quad\,\left\|(\bm u^{(m+1)}-\bm u^{(m)})(z)\right\|_{\cD(\cA^\ga)}\\
& \le C_1L_0C_3M^m\|\bm u_0\|_{L^2(\Om)}\max_{1\le k\le K}\Ga(\al_k(1-\ga))\sum_{k=1}^K J^{\al_k(1-\ga)}\left(\sum_{j=1}^K J^{\al_j(1-\ga)}\right)^{m-1}(|z|^{-\al_1\ga})\\
& =C_3M^{m+1}\|\bm u_0\|_{L^2(\Om)}\sum_{|\bm j|=m}\f{m!}{j_1!\cdots j_K!}\f{|z|^{\bm\be\cdot\bm j-\al_1\ga}}{\Ga(\bm\be\cdot\bm j+1-\al_1\ga)}
\end{align*}
in view of the definition of the constant $M$ in \eqref{eq-def-Mj}. We finish the proof of the first part of the lemma.

(ii) In the case of $\bm u_0\equiv\bm 0$, let $\ga=1$ for $p=2$ and $0\le\ga<1$ for $p\ne2$. Then the estimate \eqref{eq-est-v2} in Lemma \ref{lem-ibvp-v} directly implies
$$
\|\bm w\|_{L^p(0,T;\cD(\cA^\ga))}\le C_2\|\bm F\|_{L^p(0,T;L^2(\Om))},
$$
which is exactly \eqref{eq-est-um2} for $m=0$. For $m\in\BN$, it is readily seen from \eqref{eq-iteration} that $\bm u^{(1)}=\bm w$ and
\begin{equation}\label{eq-diff-um}
\bm u^{(m+1)}-\bm u^{(m)}=\cQ(\bm u^{(m)}-\bm u^{(m-1)})=\cdots=\cQ^m(\bm u^{(1)}-\bm u^{(0)})=\cQ^m\bm w
\end{equation}
for $m\in\BN$. Then applying Lemma \ref{lem-est-Q} to \eqref{eq-diff-um} indicates
\begin{align*}
\|(\bm u^{(m+1)}-\bm u^{(m)})(t)\|_{\cD(\cA^\ga)} & \le M^m\left(\sum_{k=1}^K J^{\al_k(1-\ga)}\right)^m\|\bm w(t)\|_{\cD(\cA^\ga)}\\
& \le M^m\sum_{|\bm j|=m}\f{m!}{j_1!\cdots j_K!}J^{\bm\be\cdot\bm j}\left(\|\bm w(t)\|_{\cD(\cA^\ga)}\right)
\end{align*}
for a.e.\! $t\in(0,T)$. Then we employ Young's convolution inequality to obtain
\begin{align*}
\|\bm u^{(m+1)}-\bm u^{(m)}\|_{L^p(0,T;\cD(\cA^\ga))} & \le M^m\|\bm w\|_{L^p(0,T;\cD(\cA^\ga))}\int_0^T\left| \sum_{|\bm j|=m}\f{m!}{j_1!\cdots j_K!}\f{t^{\bm\be\cdot\bm j-1}}{\Ga(\bm\be\cdot\bm j)} \right| \rd t\\
& \le C_2M^m\|\bm F\|_{L^p(0,T;L^2(\Om))}\sum_{|\bm j|=m}\f{m!}{j_1!\cdots j_K!}\f{T^{\bm\be\cdot\bm j}}{\Ga(\bm\be\cdot\bm j+1)}.
\end{align*}
Now Stirling's formula (e.g.\! Abramowitz and Stegun \cite[p.257]{AS72}) 
\[
\Ga(\eta)=\sqrt{2\pi}\,\e^{-\eta}\eta^{\eta-1/2}(1+O(\eta^{-1}))\quad\mbox{as }\eta\to+\infty
\]
yields
\[
\Ga(\bm\be\cdot\bm j+1)\ge\wt C\,\Ga(\un\be|\bm j|+1)=\wt C\,\Ga(\un\be m+1),\quad\un\be =\min\{\be_1,\cdots,\be_K\}
\]
with a positive constant $\wt C = \min_{y_1\ge y_2\ge 0}(\Ga(1+y_1)/\Ga(1+y_2))$.
Together with the multinomial theorem
$$
\sum_{|\bm j|=m}\f{m!}{j_1!\cdots j_K!}=K^m,
$$
it follows that
$$
\sum_{|\bm j|=m}\f{m!}{j_1!\cdots j_K!}\f{T^{\bm\be\cdot\bm j}}{\Ga(\bm\be\cdot\bm j+1)}
\le\wt C^{-1} K^m\f{\max\{1,T^{\ov\be m}\}}{\Ga(\un\be m+1)},
$$
where $\ov\be=\max\{\be_1,\dots,\be_K\}$. Collecting all the above estimates, we finally arrive at
\[
\|\bm u^{(m+1)}-\bm u^{(m)}\|_{L^p(0,T;\cD(\cA^\ga))}
\le C_2\wt C^{-1}(M K)^m\f{\max\{1,T^{\ov\be m}\}}{\Ga(\un\be m+1)} 
\|\bm F\|_{L^p(0,T;L^2(\Om))},
\]
which finishes the proof.
\end{proof}

Now we are ready to complete the proof of Theorem \ref{thm-well}. Throughout this section, by $C>0$ and $C_{T,\ga}>0$ we refer to generic constants depending only on $\Om,\bm\al,\cA,\cP$ and $\Om,\bm\al,\cA,\cP,T,\ga$, respectively, which may change from line to line. 

\begin{proof}[Proof of Theorem {\rm\ref{thm-well}(i)}]
Let $\bm u_0\in(L^2(\Om))^K$ and $\bm F\equiv\bm 0$.\medskip

\noindent{\bf Case 1 } First we assume $\cP$ is a first-order differential operator with $L^\infty$ coefficients. For clarity, we further divide the proof into three steps.\medskip

\noindent{\bf Step 1 } We claim that for $1/2\le\ga<1$, the sequence $\{\bm u^{(m)}\}$ defined by \eqref{eq-iteration} converges in $L^{1/\ga}(0,T;\cD(\cA^\ga))$, whose limit $\bm u\in L^{1/\ga}(0,T;\cD(\cA^\ga))$ is a mild solution to \eqref{eq-ibvp-u1} satisfying the estimates \eqref{eq-est-u0}--\eqref{eq-est-u1}.

Indeed, by the completeness of $\cD(\cA^\ga)$, it suffices to verify that $\{\bm u^{(m)}(z)\}$ is a Cauchy sequence in $\cD(\cA^\ga)$ in any compact subset of $S$. To this end, we pick $m,m'\in\BN$ ($m<m'$) and utilize the estimate \eqref{eq-est-um1} in Lemma \ref{lem-est-um} to estimate
\begin{align*}
\left\|(\bm u^{(m')}-\bm u^{(m)})(z)\right\|_{\cD(\cA^\ga)} & \le\sum_{\ell=m}^{m'-1}\left\|(\bm u^{(\ell+1)}-\bm u^{(\ell)})(z)\right\|_{\cD(\cA^\ga)}\\
& \le\sum_{\ell=m}^{m'-1}C_3M^{\ell+1}\|\bm u_0\|_{L^2(\Om)}\sum_{|\bm j|=\ell}\f{\ell!}{j_1!\cdots j_K!}\f{|z|^{\bm\be\cdot\bm j-\al_1\ga}}{\Ga(\bm\be\cdot\bm j+1-\al_1\ga)}\\
& \le C_3M\|\bm u_0\|_{L^2(\Om)}\,|z|^{-\al_1\ga}\sum_{\ell=m}^\infty\sum_{|\bm j|=\ell}\f{\ell!}{j_1!\cdots j_K!}\f{\prod_{k=1}^K(M\,|z|^{\be_k})^{j_k}}{\Ga(1-\al_1\ga+\bm\be\cdot\bm j)}.
\end{align*}
Now let us invoke the definition of the multinomial Mittag-Leffler function, which was first introduced by Hadid and Luchko in \cite{HL96} (see also \cite{LG99})
\[
E_{\bm\be,\be_0}(z_1,\dots,z_K):=\sum_{\ell=0}^\infty\sum_{|\bm j|=\ell}\f{\ell!}{j_1!\cdots j_K!}\f{\prod_{k=1}^K z_k^{j_k}}{\Ga(\be_0+\bm\be\cdot\bm j)}.
\]
Then we immediately see that the summation in the above estimate coincides with the remainder after the $m$-th term in the definition of $E_{\bm\be,1-\al_1\ga}(M\,|z|^{\be_1},\dots,$ $M\,|z|^{\be_K})$, which converges uniformly in any compact set of $S$. In other words, we conclude
\[
\left\|(\bm u^{(m')}-\bm u^{(m)})(z)\right\|_{\cD(\cA^\ga)}\longrightarrow0\quad\mbox{as }m,m'\to\infty\mbox{ in any compact set of }S.
\]

Thus, the sequence $\{\bm u^{(m)}(t)\}$ converges to some limit $\bm u(t)\in\cD(\cA^\ga)$ for a.e.\! $t\in(0,T)$ and $\bm u:(0,T]\longrightarrow\cD(\cA^\ga)$ is analytic. Moreover, following the argument used in the proof of \cite[Theorem 2.3]{LHY20}, we further obtain
\begin{equation}\label{eq-est-um4}
\|\bm u(t)\|_{\cD(\cA^\ga)}\le C\,t^{-\al_1\ga}\exp(\wt M t)\|\bm u_0\|_{L^2(\Om)},\quad\mbox{a.e. }t>0
\end{equation}
for $0\le \ga<1$. Therefore, it is straightforward to deduce \eqref{eq-est-u1} by
\begin{align*}
\|\bm u\|_{L^{1/\ga}(0,T;\cD(\cA^\ga))} & =\left(\int_0^T\|\bm u(t)\|_{\cD(\cA^\ga)}^{1/\ga}\,\rd t\right)^\ga\le C\exp(\wt M T)\|\bm u_0\|_{L^2(\Om)}\left(\int_0^T t^{-\al_1}\,\rd t\right)^\ga\\
& =C_{T,\ga}\|\bm u_0\|_{L^2(\Om)},\quad1/2\le\ga<1.
\end{align*}
Finally, passing $m\to\infty$ in \eqref{eq-iteration}, it is readily seen that $\bm u$ satisfies \eqref{eq-rep-u} and thus it is a mild solution to \eqref{eq-ibvp-u1} according to Definition \ref{def-mild}.

Now for $0\le\ga<1/2$, we perform $\cA^\ga$ on both sides of \eqref{eq-rep-u} and calculate directly to find
\[
\cA^\ga\bm u=\cA^\ga\cS(t)\bm u_0+\cA^\ga\cQ\bm u.
\]
In view of the estimates in Lemma \ref{lem-ibvp-v} and by using the argument in the proof of Lemma \ref{lem-est-Q}, a direct calculation yields
\begin{align*}
\|\cA^\ga\bm u(t)\|_{L^2(\Om)} & \le C_{T,\ga}\|\bm u_0\|_{L^2(\Om)}t^{-\al_1\ga}+\left\|\int_0^t\cA^{\ga-1}\cS'(t-\tau)(\cP\bm u)(\tau)\,\rd\tau\right\|_{L^2(\Om)}\\
& \le C_{T,\ga}\|\bm u_0\|_{L^2(\Om)}t^{-\al_1\ga}+\left\|\int_0^t\cA^{-1/2}\cS'(t-\tau)\cA^{\ga-1/2}(\cP\bm u)(\tau)\,\rd\tau\right\|_{L^2(\Om)}\\
& \le C_{T,\ga}\|\bm u_0\|_{L^2(\Om)}t^{-\al_1\ga}+C\sum_{k=1}^K\int_0^t (t-\tau)^{\al_k/2-1}\|u_k(\tau)\|_{\cD(\cA^{\ga})}\rd\tau\\
& \le C_{T,\ga}\|\bm u_0\|_{L^2(\Om)}t^{-\al_1\ga}+C_{T,\ga} \int_0^t (t-\tau)^{\al_K/2-1}\|\bm u(\tau)\|_{\cD(\cA^{\ga})}\rd\tau.
\end{align*}
Thus, by the generalized Gr\"onwall's inequality (e.g.\! \cite[Lemma 7.1.1]{H81}), we obtain
\begin{align*}
\|\bm u(t)\|_{\cD(\cA^\ga)} 
& \le C_{T,\ga} \|\bm u_0\|_{L^2(\Om)} t^{-\al_1\ga}+C_{T,\ga} \|\bm u_0\|_{L^2(\Om)} \int_0^t (t-\tau)^{\al_K/2-1}\tau^{-\al_1\ga} \rd\tau\\
& \le C_{T,\ga} \|\bm u_0\|_{L^2(\Om)} t^{-\al_1\ga} + C_{T,\ga} \|\bm u_0\|_{L^2(\Om)} t^{\al_K/2-\al_1\ga}.
\end{align*}
Finally, we can obtain
\[
\|\bm u\|_{L^{1/\ga}(0,T;\cD(\cA^\ga))}\le C_{T,\ga}\|\bm u_0\|_{L^2(\Om)},\quad 0\le\ga<1.
\]

\noindent{\bf Step 2 } Next we demonstrate the uniqueness of the mild solution and the convergence \eqref{eq-IC-u}.

For the uniqueness, assume that $\bm v$ is another mild solution to \eqref{eq-ibvp-u1}. Taking difference between the integral equations \eqref{eq-rep-u} for $\bm u$ and $\bm v$, we obtain
\[
\bm u-\bm v=\cQ(\bm u-\bm v).
\]
Then for $\ga\in[1/2,1)$ applying Lemma \ref{lem-est-Q} with $m=1$ immediately yields
\begin{align*}
\|(\bm u-\bm v)(t)\|_{\cD(\cA^\ga)} & \le M\sum_{k=1}^K\int_0^t\|(\bm u-\bm v)(\tau)\|_{\cD(\cA^\ga)}(t-\tau)^{\al_k(1-\ga)-1}\,\rd\tau\\
& \le C\int_0^t\|(\bm u-\bm v)(\tau)\|_{\cD(\cA^\ga)}(t-\tau)^{\al_K(1-\ga)-1}\,\rd\tau,\ 0<t<T.
\end{align*}
Thus we can employ Gr\"onwall's inequality with a weakly singular kernel (e.g.\! \cite[Lemma 7.1.1]{H81}) to conclude $\bm u\equiv\bm v$.

Now we turn to the convergence issue \eqref{eq-IC-u}. We have
\[
\bm u(t)-\bm u_0=(\bm w(t)-\bm u_0)+\cQ\bm u(t),
\]
where $\lim_{t\to0}\|\bm w(t)-\bm u_0\|_{L^2(\Om)}=0$ is guaranteed by \eqref{eq-IC-v} in Lemma \ref{lem-ibvp-v}. For $\cQ\bm u(t)$, again we apply Lemma \ref{lem-est-Q} with $m=1$ and employ 
\eqref{eq-est-um4} with $\ga=0$ to dominate
\begin{align}
\|\cQ\bm u(t)\|_{L^2(\Om)} & \le M\sum_{k=1}^K\int_0^t\|\bm u(\tau)\|_{L^2(\Om)}\f{(t-\tau)^{\al_k-1}}{\Ga(\al_k)}\,\rd\tau\nonumber\\
& \le C_{T,\ga}K\|\bm u_0\|_{L^2(\Om)}\int_0^t \exp(\wt M\,\tau)(t-\tau)^{\al_K-1}\,\rd\tau\nonumber\\
& \le C_{T,\ga}K\exp(\wt M\,t)\|\bm u_0\|_{L^2(\Om)}\int_0^t (t-\tau)^{\al_K-1}\,\rd\tau\nonumber\\
& =C_{T,\ga}K\exp(\wt M\,t)\|\bm u_0\|_{L^2(\Om)}t^{\al_K}\longrightarrow0\quad(t\to0).\label{eq-IC-y}
\end{align}
Thus \eqref{eq-IC-u} is verified.\medskip

\noindent{\bf Step 3 } Finally, we improve the $\bm x$-regularity of $\bm u$ to show $\bm u\in L^1(0,T;\cD(\cA))$ and the estimate \eqref{eq-est-u1} for $\ga=1$. Recall that now
\[
\bm u=\bm w+\bm y \quad \text{with } \bm w = \cS(\cdot) \bm u_0, \ \bm y:=\cQ\bm u.
\]
According to the estimate \eqref{eq-est-v1} in Lemma \ref{lem-ibvp-v}(i), we immediately see
\begin{align}
\|\bm w\|_{L^{1/\ga}(0,T;\cD(\cA^\ga))} & =\left(\int_0^T\|\bm w(t)\|_{\cD(\cA^\ga)}^{1/\ga}\,\rd t\right)^\ga\le C_{T,\ga}\|\bm u_0\|_{L^2(\Om)}\left(\int_0^T t^{-\al_1}\,\rd t\right)^\ga\nonumber\\
& \le C_{T,\ga}\|\bm u_0\|_{L^2(\Om)},\quad 0\le\ga\le1.\label{eq-est-w}
\end{align}
On the other hand, by the definition of $\cQ$ and Lemma \ref{lem-ibvp-v}(ii), it reveals that $\bm y$ satisfies the initial-boundary value problem
\begin{equation}\label{eq-ibvp-y}
\begin{cases}
(\pa_t^{\bm\al}+\cA)\bm y=\cP\bm u & \mbox{in }\Om\times(0,T),\\
\bm y=\bm 0 & \mbox{in }\Om\times\{0\},\\
\bm y=\bm 0 & \mbox{on }\pa\Om\times(0,T).
\end{cases}
\end{equation}
Here the initial condition is understood in the sense of \eqref{eq-IC-y}. Notice that in Step 1, we already obtained $\bm u\in L^{1/\ga}(0,T;\cD(\cA^\ga))$ along with \eqref{eq-est-u1} for $0\le\ga<1$, and in particular
\begin{gather*}
\bm u\in L^2(0,T;\cD(\cA^{1/2}))=(L^2(0,T;H_0^1(\Om)))^K,\\
\|\bm u\|_{L^2(0,T;H_0^1(\Om))}\le C_{T,\ga}\|\bm u_0\|_{L^2(\Om)}.
\end{gather*}
Then for $\ga =1$, we can employ Lemma \ref{lem-ibvp-v}(ii) with $p=2$ to derive
\begin{align*}
\|\bm y\|_{L^1(0,T;\cD(\cA))} & \le T^{1/2}\|\bm y\|_{L^2(0,T;\cD(\cA))}\le C_2T^{1/2}\|\cP\bm u\|_{L^2(0,T;L^2(\Om))}\\
& \le C_{T,\ga}\|\bm u\|_{L^2(0,T;H_0^1(\Om))}\le C_{T,\ga}\|\bm u_0\|_{L^2(\Om)}.
\end{align*}
Finally, combining the above inequality with \eqref{eq-est-w} leads to $\bm u=\bm w+\bm y\in L^1(0,T;\cD(\cA))$ and the estimate \eqref{eq-est-u1} for $\ga=1$.\medskip

\noindent{\bf Case 2 } Now we assume $\cP$ is a first-order differential operator with $W^{1,\infty}$ coefficients in $\bm x$. Indeed, the key estimate
\[
\|\cQ^m\bm v(t)\|_{\cD(\cA^\ga)}\le M^m\left(\sum_{k=1}^K J^{\al_k/2-1}\right)^m\|\bm v(t)\|_{\cD(\cA^\ga)}
\]
analogous to that in Lemma \ref{lem-est-Q} still holds true in the case of $1/2<\ga\le1$. Repeating the same argument as that in Step 1, we can easily conclude that the mild solution $\bm u$ belongs to $L^{1/\ga}(0,T;\cD(\cA^\ga))$ and satisfies \eqref{eq-est-u0} also for $0\le\ga\le1$. This completes the proof of Theorem \ref{thm-well}(i).
\end{proof}

\begin{proof}[Proof of Theorem {\rm\ref{thm-well}(ii)}]
Let $\bm u_0\equiv\bm 0$ and $\bm F\in(L^p(0,T;L^2(\Om)))^K$ with $p\in[1,+\infty]$. Throughout this proof, we restrict $0\le\ga<1$ for $p\ne2$ and allow $\ga=1$ for $p=2$.

We argue rather similarly as the previous proof and we claim that the sequence $\{\bm u^{(m)}\}$ defined by \eqref{eq-iteration} converges in $L^p(0,T;\cD(\cA^\ga))$, whose limit $\bm u\in L^p(0,T;\cD(\cA^\ga))$ is the unique mild solution to \eqref{eq-ibvp-u1} satisfying the estimate \eqref{eq-est-u2}.

Indeed, we pick $m,m'\in\BN$ ($m<m'$) and turn to the estimate \eqref{eq-est-um2} in Lemma \ref{lem-est-um} to estimate
\begin{align*}
\|\bm u^{(m')}-\bm u^{(m)}\|_{L^p(0,T;\cD(\cA^\ga))} & \le\sum_{\ell=m}^{m'-1}\|\bm u^{(\ell+1)}-\bm u^{(\ell)}\|_{L^p(0,T;\cD(\cA^\ga))}\\
& \le\|\bm F\|_{L^p(0,T;L^2(\Om))}\sum_{\ell=m}^{m'-1}C(M K)^\ell\f{\max\{1,T^{\ov\be\ell}\}}{\Ga(\un\be\ell+1)}\\
& \le C\|\bm F\|_{L^p(0,T;L^2(\Om))}\sum_{\ell=m}^\infty\f{(M K\max\{1,T^{\ov\be}\})^\ell}{\Ga(\un\be\ell+1)}.
\end{align*}
Again noticing that the above summation coincides with the remainder after the $m$-th term in the Mittag-Leffler function $E_{\un\be,1}(M K\max\{1,T^{\ov\be}\})$, we can easily conclude
\[
\|\bm u^{(m')}-\bm u^{(m)}\|_{L^p(0,T;\cD(\cA^\ga))}\longrightarrow0\quad\mbox{as }m,m'\to\infty
\]
due to the uniform convergence of $E_{\un\be,1}(M K\max\{1,T^{\ov\be}\})$. Therefore, $\{\bm u^{(m)}\}$ is a Cauchy sequence in $L^p(0,T;\cD(\cA^\ga))$ and thus converges to some limit $\bm u\in L^p(0,T;\cD(\cA^\ga))$, which is obviously a mild solution to \eqref{eq-ibvp-u1}. Moreover, similarly we deduce
\begin{align*}
\|\bm u\|_{L^p(0,T;\cD(\cA^\ga))} & \le\sum_{\ell=0}^\infty\|\bm u^{(\ell+1)}-\bm u^{(\ell)}\|_{L^p(0,T;\cD(\cA^\ga))}\\
& \le C\|\bm F\|_{L^p(0,T;L^2(\Om))}\sum_{\ell=0}^\infty\f{(M K\max\{1,T^{\ov\be}\})^\ell}{\Ga(\un\be\ell+1)}\\
& =C\|\bm F\|_{L^p(0,T;L^2(\Om))}E_{\un\be,1}(M K\max\{1,T^{\ov\be}\})\\
& \le C_{T,\ga}\|\bm F\|_{L^p(0,T;L^2(\Om))},
\end{align*}
where we used Lemma \ref{lem-ML} again to treat $E_{\un\be,1}(M K\max\{1,T^{\ov\be}\})$. This completes the proof of \eqref{eq-est-u2}. The uniqueness of $\bm u$ follows from identically the same argument as the proof of Theorem \ref{thm-well}(i) and this completes the proof of Theorem \ref{thm-well}(ii).
\end{proof}

\section{Proof of Theorem \ref{thm-asymp}}\label{sec-asymp}

Throughout this section, we assume \eqref{eq-weak} and consider the long-time asymptotic behavior of the solution to the following initial-boundary value problem
\begin{equation}\label{equ-uv-infty}
\begin{cases}
\pa_t^{\bm\al}\bm u+\cA\bm u=\bm C\bm u & \mbox{in }\Om\times\BR_+,\\
\bm u=\bm u_0 & \mbox{in }\Om\times\{0\},\\
\bm u=\bm0 & \mbox{on }\pa\Om\times\BR_+,
\end{cases}
\end{equation}
where we recall $\bm C=(c_{k\ell})_{1\le k,\ell\le K}\in(L^\infty(\Om))^{K\times K}$. Based on the results in the above section, we know that the solution $\bm u$ to \eqref{eq-ibvp-u1} uniquely exists in $H_0^1(\Om)\cap H^{2\ga}(\Om)$ for any $t\in(0,T]$ and admits the estimate
\begin{equation}\label{esti-H2}
\|\bm u(t)\|_{\cD(A^\ga)}\le C_T\,t^{-\al_1\ga}\|\bm u_0\|_{L^2(\Om)},\quad0<t\le T
\end{equation}
with $\ga\in[0,1)$. Thus the asymptotic behavior as $t\to0$ is related to the largest order of the fractional derivatives. However, we know nothing about the long-time asymptotic behavior from the above estimate except that the constant in \eqref{esti-H2} depends on $T$, which may tends to infinity as $T\to+\infty$. Therefore, we shall employ the Laplace transform
\[
\wh f(s):=\int_{\BR_+}f(t)\,\e^{-s t}\,\rd t
\]
to investigate the long-time asymptotic behavior of the solution to \eqref{equ-uv-infty}.

Before giving the proof of Theorem \ref{thm-asymp}, we make some necessary settings. We introduce the sector
$$
S_\te:=\{s\in\BC\setminus\{0\}\mid|\arg s|<\te\}
$$
with an angel $\te\in(\f\pi2,\min\{\f\pi{2\al_1},\pi\})$. We apply the Laplace transform to \eqref{equ-uv-infty} and use the formula
$$
\wh{\pa_t^\al f}(s)=s^\al\wh f(s)-s^{\al-1}f(0),\quad0<\al<1
$$
to derive the boundary value problem for a coupled elliptic system
\begin{equation}\label{equ-lap-uv}
\begin{cases}
(\cA+s^{\bm\al}-\bm C)\wh{\bm u}(\,\cdot\,;s)=s^{-1}s^{\bm\al}\bm u_0 & \mbox{in }\Om,\\
\wh{\bm u}(\,\cdot\,;s)=\bm0 & \mbox{on }\pa\Om
\end{cases}
\end{equation}
with a parameter $s\in\{\rRe\,s>s_0\}$. Here we abbreviate
\[
s^{\bm\al}:=\diag(s^{\al_1},\dots,s^{\al_K})
\]
and $s_0\ge C$ with the constant $C$ in \eqref{eq-est-u0}. 
As before, we simply write $\wh{\bm u}(\,\cdot\,;s)$ as $\wh{\bm u}(s)$.

We check at once that $\wh{\bm u}:\{\rRe\,s>s_0\}\longrightarrow(L^2(\Om))^K$ is analytic, which is clear from the property of the Laplace transform. Moreover, we claim that $\wh{\bm u}(s)$ ($\rRe\,s>s_0$) can be analytically extended to the sector $S_\te$. Namely, the following lemma holds.

\begin{lem}\label{lem-analy-lap-uv}
Under the same assumptions in Theorem $\ref{thm-asymp},$ for any $s\in S_\te$ the boundary value problem \eqref{equ-lap-uv} admits a unique weak solution $\wh{\bm u}(s)\in(H^2(\Om))^K$ and $\wh{\bm u}:S_\te\longrightarrow(H^2(\Om))^K$ is analytic. Moreover$,$ there exists a constant $C'>0$ depending only on $d,\Om,\te,\bm\al,\bm C$ and $\cA$ such that
\begin{equation}\label{esti-lap-uv}
\|\wh{\bm u}(s)\|_{H^2(\Om)}\le C'\|\bm u_0\|_{L^2(\Om)}\left(\sum_{k,\ell=1}^K|s|^{\al_k+\al_\ell-1}+\sum_{k=1}^K|s|^{\al_k-1}\right),\quad\forall\,s\in S_\te. 
\end{equation}
\end{lem}

\begin{proof}
Firstly for any fixed $s\in S_\te$, we define a bilinear form
\[
B[\,\cdot\,,\,\cdot\,;s]:(H_0^1(\Om))^K\times(H_0^1(\Om))^K\longrightarrow\BC
\]
by
$$
B[\bm\psi,\bm\vp;s] :=\int_\Om\left(\sum_{k=1}^K\bm A_k\nb\psi_k\cdot\nb\vp_k+(s^{\bm\al}-\bm C)\bm\psi\cdot\bm\vp\right)\rd\bm x,
$$
where $\bm\psi=(\psi_1,\psi_2,\cdots,\psi_K)^\T,\bm\vp=(\vp_1,\vp_2,\cdots,\vp_K)^\T\in(H_0^1(\Om))^K$. 
By using the Poincar\'e inequality and H\"older's inequality, we have
\[
|B[\bm\psi,\bm\vp;s]|\le C(s)\int_\Om\left(\sum_{k=1}^K|\nb\psi_k||\nb\vp_k|+|\bm\psi||\bm\vp|\right)\rd\bm x\le C(s)\|\bm\psi\|_{H^1(\Om)}\|\bm\vp\|_{H^1(\Om)}.
\]
Here the constant $C(s)>0$ depends on $s$. Furthermore, taking $\bm\vp=\bm\psi$ implies
$$
B[\bm\psi,\bm\psi;s]=\int_\Om\left(\sum_{k=1}^K\bm A_k\nb\psi_k\cdot\nb\psi_k+(s^{\bm\al}-\bm C)\bm\psi\cdot\bm\psi\right)\rd\bm x,
$$
and hence
\begin{align*}
\rRe(B[\bm\psi,\bm\psi;s]) & =\int_\Om\left\{\sum_{k=1}^K\left(\bm A_k\nb(\rRe\,\psi_k)\cdot\nb(\rRe\,\psi_k)+\bm A_k\nb(\rIm\,\psi_k)\!\cdot\!\nb(\rIm\,\psi_k)+(\rRe\,s^{\al_k})|\psi_k|^2\right)\right.\\
& \qquad\quad\,\,-\bm C(\rRe\,\bm\psi)\cdot(\rRe\,\bm\psi)-\bm C(\rIm\,\bm\psi)\cdot(\rIm\,\bm\psi)\bigg\}\,\rd\bm x.
\end{align*}
Since the matrix $\bm C$ is assumed to be negative semi-definite, it follows that 
$$
-\bm C(\rRe\,\bm\psi)\cdot(\rRe\,\bm\psi)\ge0,\quad-\bm C(\rIm\,\bm\psi)\cdot(\rIm\,\bm\psi)\ge0\quad\mbox{in }\Om.
$$
Further noticing that $\rRe\,s^{\al_k}=r^{\al_k}\cos\al_k\rho>0$ ($k=1,2,\cdots,K$) for $s=r\,\e^{\ri\rho}\in S_\te$ with $\te\in(\f\pi2,\min\{\f\pi{2\al_1},\pi\})$, we employ the above inequality and the ellipticity of $\cA_k$ to deduce
$$
\rRe(B[\bm\psi,\bm\psi;s])\ge \kappa \int_\Om\sum_{k=1}^K\left(|\nb(\rRe\,\psi_k)|^2+|\nb(\rIm\,\psi_k)|^2\right)\rd\bm x\ge C_0\|\bm\psi\|_{H^1(\Om)}^2
$$
for some positive constant $C_0$, where we used the Poincar\'e inequality in the last inequality. Consequently, the Lax-Milgram theorem yields that for any $s\in S_\te$, there exists a unique $\wh{\bm u}(s)\in(H_0^1(\Om))^K$ such that
$$
B[\wh{\bm u}(s),\bm\vp;s]=s^{-1}\int_\Om s^{\bm\al}\bm u_0\cdot\bm\vp\,\rd\bm x,\quad\forall\,\bm\vp\in(H_0^1(\Om))^K,
$$
which implies
\[
C_0\|\wh{\bm u}(s)\|_{H^1(\Om)}^2\le|B[\wh{\bm u}(s),\wh{\bm u}(s);s]|\le \sum_{k=1}^K|s|^{\al_k-1}\|u_0^{(k)}\|_{L^2(\Om)}\|\wh{\bm u}(s)\|_{H^1(\Om)}
\]
in view of H\"older's inequality, and consequently
\begin{equation}\label{esti-u-H1}
\|\wh{\bm u}(s)\|_{H^1(\Om)}\le C'\sum_{k=1}^K|s|^{\al_k-1}\|u_0^{(k)}\|_{L^2(\Om)},\quad s\in S_\te.
\end{equation}
Furthermore, since $\bm u_0\in(L^2(\Om))^K$, then by the regularity estimate for elliptic equations (e.g.\! \cite{E98}), we see that $\wh{\bm u}(s)\in(H^2(\Om))^K$ with the estimate
\begin{align*}
\|\wh{\bm u}(s)\|_{H^2(\Om)} & \le C'\left\|(s^{\bm\al}-\bm C)\wh{\bm u}(s)-s^{-1}s^{\bm\al}\bm u_0\right\|_{L^2(\Om)}\\
& \le C'\sum_{k=1}^K|s|^{\al_k}\|\wh{u_k}(s)\|_{L^2(\Om)}+C'\|\bm C\wh{\bm u}(s)\|_{L^2(\Om)}+C'\sum_{k=1}^K|s|^{\al_k-1}\|u_0^{(k)}\|_{L^2(\Om)}\\
& \le C'\|\bm u_0\|_{L^2(\Om)}\left(\sum_{k,\ell=1}^K|s|^{\al_k+\al_\ell-1}+\sum_{k=1}^K|s|^{\al_k-1}\right),\quad s\in S_\te.
\end{align*}
Similar to the argument used in the proof of \cite[Theorem 0.1]{Pruss}, we can prove that the solution $\wh{\bm u}(s)$ is analytic in $s\in S_\te$. The proof of Lemma \ref{lem-analy-lap-uv} is completed.
\end{proof}

Now we can proceed to complete the proof of Theorem \ref{thm-asymp}. From Lemma \ref{lem-analy-lap-uv}, we see that there exists a unique analytical extension of $\wh{\bm u}(s)$ from $\{\rRe\,s>s_0\}$ to $S_\te$. By the same notation we denote this extension.

By Fourier-Mellin formula (e.g.\! \cite{S91}) for the inverse Laplace transform, we have
$$
\bm u(t)=\f1{2\pi\ri} \int_{s_0-\ri\infty}^{s_0+\ri\infty} \wh{\bm u}(s)\,\e^{s t}\,\rd s.
$$
Since $\wh{\bm u}(s)$ is analytic in the sector $S_\te$, it follows from the residue theorem (e.g.\! \cite{R86}) that for $t>0$, $\bm u(t)$ can be represented by an integral on the contour
\[
\ga(\ve,\te):=\{s\in\BC\mid\arg s=\te,\ |s|\ge\ve\}\cup\{s\in\BC\mid|\arg s|\le\te,\ |s|=\ve\}
\]
with $\ve>0$. In other words, there holds
$$
\bm u(t)=\f1{2\pi\ri}\int_{\ga(\ve,\te)}\wh{\bm u}(s)\,\e^{s t}\,\rd s,
$$
where the shift in the line of integration is justified by the estimate \eqref{esti-lap-uv}. 
Moreover, again from the estimate \eqref{esti-lap-uv}, we can pass $\ve\to0$ to obtain
\[
\bm u(t)=\f1{2\pi\ri}\int_{\ga(0,\te)}\wh{\bm u}(s)\,\e^{s t}\,\rd s.
\]
Then we conclude from \eqref{esti-lap-uv} that
\begin{align*}
\|\bm u(t)\|_{H^2(\Om)} & \le C\int_{\BR_+}\left(\|\wh{\bm u}(r\,\e^{\ri\te})\|_{H^2(\Om)}+\|\wh{\bm u}(r\,\e^{-\ri\te})\|_{H^2(\Om)}\right)\e^{r t\cos\te}\,\rd r\\
& \le C\|\bm u_0\|_{L^2(\Om)}\int_{\BR_+}\e^{r t\cos\te}\left(\sum_{k,\ell=1}^K r^{\al_k+\al_\ell-1}+\sum_{k=1}^K r^{\al_k-1}\right)\,\rd r
\end{align*}
and thus
$$
\|\bm u(t)\|_{H^2(\Om)}\le C\|\bm u_0\|_{L^2(\Om)}\left(\sum_{k,\ell=1}^K t^{-\al_k-\al_\ell}+\sum_{k=1}^K t^{-\al_k}\right),\quad t>0
$$
in view of $\cos\te<0$ since $\f\pi2 <\te<\min\{\f\pi{2\al_1},\pi\}$. 
Moreover, in particular we obtain
$$
\|\bm u(t)\|_{H^2(\Om)}\le C\,t^{-\al_K}\|\bm u_0\|_{L^2(\Om)},\quad t\ge t_0
$$
for any fixed $t_0>0$. Here $C>0$ is a new generic constant depending also on $t_0$.

Finally, let us turn to proving the sharpness of the decay rate $t^{-\al_K}$ by contradiction. If the decay rate $t^{-\al_K}$ is not sharp, then there exists a nonnegative and bounded function $f$ in $[0,+\infty)$ satisfying $\lim_{t\to+\infty}f(t)=0$ and
$$
\|\bm u(t)\|_{H^2(\Om)}\le C\,t^{-\al_K}f(t)\|\bm u_0\|_{L^2(\Om)},\quad t>0.
$$
Then we see that the Laplace transform $\wh{\bm u}(s)$ admits the following estimate
$$
\|\wh{\bm u}(s)\|_{H^2(\Om)}\le C\,g(s) s^{\al_K-1}\|\bm u_0\|_{L^2(\Om)},\quad 0<s<1,
$$
where 
$$
g(s):=s^{1-\al_K}\wh{(t^{-\al_K}f)}(s) = s^{1-\al_K}\int_0^\infty t^{-\al_K}f(t)\,\e^{-s t}\,\rd t,\quad s\ge 0.
$$
Moreover, we claim that 
\begin{equation}\label{eq-g}
g\mbox{ is bounded on $[0,1]$ and }\lim_{s\to0}g(s)=0, 
\end{equation}
which will be verified in Appendix \ref{sec-app}. Then immediately we see that 
$$
s^{1-\al_K}\cA\wh{\bm u}(s)\longrightarrow\bm 0\quad\mbox{as }s\to0.
$$
Now multiplying $s^{1-\al_K}$ on both sides of \eqref{equ-lap-uv} and passing $s\to0$ imply
$$
\bm0=\lim_{s\to0}s^{-\al_K}s^{\bm\al}\bm u_0\quad\mbox{in }\Om
$$
and especially $u_0^{(K)}\equiv 0$ in $\Om$, which is a contradiction. This completes the proof of Theorem \ref{thm-asymp}.

\section{Proof of Theorem \ref{thm-ip}}\label{sec-alpha}

In this section, we prove the unique determination of the orders by the observation of $u_{k_0}$ at $\{\bm x_0\}\times(0,T)$. 
By noting the time analyticity of the solution to \eqref{eq-ibvp-u1}, we can uniquely extend the solution $\bm u$ from $(0,T)$ to $\BR_+$. Thus, by applying the Laplace transform on both sides of the equation \eqref{eq-ibvp-u1} with $\bm\al=(\al_1,\dots,\al_K)^{\T}$ and $\bm\be=(\be_1,\dots,\be_K)^{\T}$, respectively and recalling the notation $s^{\bm\al}=\diag(s^{\al_1},\dots,s^{\al_K})$, we find
\[
\begin{cases}
(\cA-\bm C+s^{\bm\al})\wh{\bm u}(s)=s^{-1}s^{\bm\al}\bm u_0 & \mbox{in }\Om,\\
\wh{\bm u}(s)=\bm0 & \mbox{on }\pa\Om,
\end{cases}\quad\begin{cases}
(\cA-\bm C+s^{\bm\be})\wh{\bm v}(s)=s^{-1}s^{\bm\be}\bm u_0 & \mbox{in }\Om,\\
\wh{\bm v}(s)=\bm0 & \mbox{on }\pa\Om
\end{cases}
\]
for $s>0$. By taking the difference between the above problems, it turns out that $(\wh{\bm u}-\wh{\bm v})(s)$ satisfies
\begin{equation}\label{equ-v}
\begin{cases}
(\cA-\bm C+s^{\bm\al})(\wh{\bm u}-\wh{\bm v})(s)=s^{-1}(s^{\bm\al}-s^{\bm\be}) (\bm u_0-s\,\wh{\bm v}(s)) & \mbox{in }\Om,\\
(\wh{\bm u}-\wh{\bm v})(s)=\bm0 & \mbox{on }\pa\Om.
\end{cases}
\end{equation}
We prove by contradiction. 
Let us assume that $\bm\al\ne\bm\be$. Then there exists an index $k_1\in\{1,\dots,K\}$ such that
$$
\al_{k_1}\ne\be_{k_1},\quad\al_k=\be_k,\quad k>k_1.
$$ 
In other words, $\al_{k_1}$ is the smallest order in $\bm\al$ such that $\al_k\ne\be_k$. Without loss of generality, we assume $\al_{k_1}<\be_{k_1}$. 

Moreover, we introduce the following auxiliary function $\bm w$ defined by 
$$
\bm w(s):=\f{s\,(\wh{\bm u}-\wh{\bm v})(s)}{s^{\al_{k_1}}-s^{\be_{k_1}}},\quad s>0,
$$
which shares the same sign as that of $(\wh{\bm u}-\wh{\bm v})(s)$ for each component. Multiplying both sides of the governing equation in \eqref{equ-v} by $s/(s^{\al_{k_1}}-s^{\be_{k_1}})$, we see that $\bm w$ satisfies the boundary value problem
\begin{equation}\label{equ-w}
\begin{cases}
(\cA-\bm C+s^{\bm\al})\bm w(s)=\bm\Si(\bm u_0-s\,\wh{\bm v}(s)) & \mbox{in }\Om,\\
\bm w(s)=\bm0 & \mbox{on }\pa\Om,
\end{cases}
\end{equation}
where
\[
\bm\Si:=\diag\left(\f{s^{\al_1}-s^{\be_1}}{s^{\al_{k_1}}-s^{\be_{k_1}}},\dots,\f{s^{\al_{k_1-1}}-s^{\be_{k_1-1}}}{s^{\al_{k_1}}-s^{\be_{k_1}}},1,0,\dots,0\right).
\]
From the property of the Laplace transform, it follows that $\wh{\bm u}(s)$ and $\wh{\bm v}(s)$ are analytic with respect to $s>0$, so that $\bm w(s)$ is continuous for $s>0$ in view of its definition.

Next we discuss the limit of $\bm w(s)$ as $s\to0$. We claim that 
$$
\lim_{s\to0}\bm w(s)=\bm w_0\quad\mbox{in }(H^2(\Om))^K,
$$
where $\bm w_0=(w_0^{(1)},\dots,w_0^{(K)})^\T$ solves the following boundary value problem:
\begin{equation}\label{equ-w0}
\begin{cases}
(\cA-\bm C)\bm w_0=\bm D\bm u_0 & \mbox{in }\Om,\\
\bm w_0=\bm 0 & \mbox{on }\pa\Om.
\end{cases}
\end{equation}
Here $\bm D=(d_{k\ell})_{1\le k,\ell\le K}$ is a diagonal matrix with
$$
d_{k k}=\begin{cases}
1, & (k\in I_1:=\{k<k_1\mid\al_k=\al_{k_1}\})\mbox{ or }(k=k_1),\\
0, & (k\in I_2:=\{k<k_1\mid\al_k\ne\al_{k_1}\})\mbox{ or }(k>k_1).
\end{cases}
$$
In fact, letting $\bm z(s):=\bm w(s)-\bm w_0$ ($s>0$), we take the difference between \eqref{equ-w} and \eqref{equ-w0} to obtain
\begin{equation}\label{equ-z}
\left\{\!\begin{alignedat}{2}
& (\cA-\bm C+s^{\bm\al})\bm z(s) = \bm H(s)  & \quad & \mbox{in }\Om,\\
&\bm z(s)=\bm0 & \quad & \mbox{on }\pa\Om
\end{alignedat}\right.
\end{equation}
where
\begin{equation}\label{eq-def-H}
\bm H(s):=(\bm\Si-\bm D)\bm u_0-s\bm\Si\wh{\bm v}(s)-s^{\bm\al}\bm w_0.
\end{equation}
By the conditions \eqref{eq-cond-C1}--\eqref{eq-cond-C2} in Theorem \ref{thm-ip}, we can apply the maximum principle for coupled elliptic equations (e.g.\! \cite{FM90}) to conclude that the operator $\mathcal{A}-\bm C+s^{\bm\alpha}$ is invertible for any $s>0$. Thus, the regularity estimate for elliptic equations (e.g.\! \cite{E98}) implies
\begin{align*}
\|\bm z(s)\|_{H^2(\Om)}\le C\|\bm H(s)\|_{L^2(\Om)}. 
\end{align*}
Here the generic constant $C>0$ also depends on the $L^\infty$-norm of the coefficients in \eqref{equ-z}. However, henceforth we consider only small $s$ (e.g.\! $0<s\le 1$) so that the constant $C$ is independent of $s$. It remains to check that $\|\bm H(s)\|_{L^2(\Om)}$ vanishes as $s\to0$. By noting the definition of the diagonal matrix $\bm D$, we find that the $k$-th entry of the diagonal matrix $\bm\Si-\bm D$ in \eqref{eq-def-H} reads
\[
\left\{\!\begin{alignedat}{2}
& \f{s^{\al_k}-s^{\be_k}}{s^{\al_{k_1}}-s^{\be_{k_1}}}-1=\f{s^{\be_{k_1}}-s^{\be_k}}{s^{\al_{k_1}}-s^{\be_{k_1}}}=\f{s^{\be_{k_1}-\al_{k_1}}-s^{\be_k-\al_{k_1}}}{1-s^{\be_{k_1}-\al_{k_1}}}, & \quad & k\in I_1,\\
& \f{s^{\al_k}-s^{\be_k}}{s^{\al_{k_1}}-s^{\be_{k_1}}}=\f{s^{\al_k-\al_{k_1}}-s^{\be_k-\al_{k_1}}}{1-s^{\be_{k_1}-\al_{k_1}}}, & \quad & k\in I_2,\\
& 0, & \quad & k\ge k_1.
\end{alignedat}\right.
\]
Then the norm of the first term in $\bm H(s)$ admits 
\[
\|(\bm\Si-\bm D)\bm u_0\|_{L^2(\Om)}\le\|\bm u_0\|_{L^2(\Om)}\left(\sum_{k\in I_1}\f{s^{\be_{k_1}-\al_{k_1}}-s^{\be_k-\al_{k_1}}}{1-s^{\be_{k_1}-\al_{k_1}}}+\sum_{k\in I_2}\f{s^{\al_k-\al_{k_1}}-s^{\be_k-\al_{k_1}}}{1-s^{\be_{k_1}-\al_{k_1}}}\right).
\]
Moreover, in view of the definition of $\bm w$, 
we estimate the norm of the second term in $\bm H(s)$ as follows
\begin{align*}
\|-s\bm\Si\wh{\bm v}(s)\|_{L^2(\Om)} & \le\|s\wh{\bm v}(s)\|_{L^2(\Om)}\sum_{k=1}^{k_1}\f{s^{\al_k-\al_{k_1}}-s^{\be_k-\al_{k_1}}}{1-s^{\be_{k_1}-\al_{k_1}}}\\
& \le C\|\bm u_0\|_{L^2(\Om)}\sum_{\ell=1}^K s^{\al_\ell}\sum_{k=1}^{k_1}\f{s^{\al_k-\al_{k_1}}-s^{\be_k-\al_{k_1}}}{1-s^{\be_{k_1}-\al_{k_1}}}.
\end{align*}
Here in the last inequality we used the estimate \eqref{esti-u-H1} for $\wh{\bm v}$. 
Noting that $\bm w_0$ is the solution to \eqref{equ-w0}, again by the regularity estimate for elliptic equations, we readily find that 
$$
\left\|-s^{\bm\al}\bm w_0\right\|_{L^2(\Om)}\le C\|\bm u_0\|_{L^2(\Om)}\sum_{k=1}^K s^{\al_k}. 
$$
Combining the above estimates and noticing the fact that
$$
\al_{k_1}<\be_{k_1}\le\be_k,\ k\in I_1\cup I_2, \quad\al_{k_1}<\al_k,\ k\in I_2,
$$
we obtain $\lim_{s\to0}\|\bm H(s)\|_{L^2(\Om)}= 0$ and hence $\lim_{s\to0}\|\bm z(s)\|_{H^2(\Om)}= 0$.

From Lemma \ref{lem-smp} in Appendix \ref{sec-app}, we conclude that $\bm w_0>\bm0$ in $\Om$, that is, $w_0^{(k)}>0$ in $\Om$ for all $k=1,\dots,K$. Then at the observation point $\bm x_0\in\Om$, from the relation $\bm w_0=\lim_{s\to0}\bm w(s)$ in $(H^2(\Om))^K\subset(C(\ov\Om))^K$, we have
$$
\lim_{s\to0}w_k(\bm x_0;s)=w_0^{(k)}(\bm x_0)>0,\quad k=1,\dots,K.
$$
This indicates that we can choose a small $\ve>0$ such that $w_k(\bm x_0;s)>0$ for any $s\in(0,\ve)$ and $k=1,\dots,K$, which implies that 
$$
\wh u_k(\bm x_0;s)-\wh v_k(\bm x_0;s)>0,\quad 0<s<\ve,\ k=1,\dots,K.
$$
Therefore, we have $\wh{\bm u}(\bm x_0;s)>\wh{\bm v}(\bm x_0;s)$ for all $s\in (0,\ve)$, that is, 
$$
\wh u_k(\bm x_0;s)>\wh v_k(\bm x_0;s),\quad0<s<\ve,\ k=1,\dots,K.
$$
This yields a contradiction since $u_{k_0}=v_{k_0}$ at $\{\bm x_0\}\times(0,T)$ implies $u_{k_0}=v_{k_0}$ at $\{\bm x_0\}\times\BR_+$ by the time analyticity, and hence $\wh{u_{k_0}}(\bm x_0;s)=\wh{v_{k_0}}(\bm x_0;s)$ for any $s>0$. The proof of Theorem \ref{thm-ip} is completed.

\section{Conclusions}\label{sec-rem}

In this paper, we investigate the fundamental properties of the initial-boundary value problem \eqref{eq-ibvp-u0} for a moderately coupled system of linear time-fractional diffusion equations. In view of the unique existence, limited smoothing property and long-time asymptotic behavior, it reveals that the coupled system inherits almost the same properties as those of a single equation, especially those of a multi-term time-fractional diffusion equation. Employing the coupling effect, we also established the uniqueness for an inverse problem on the simultaneous determination of multiple parameters by the observation of a single component of the solution. These indicate the similarity and difference between a single equation and a coupled system of time-fractional PDEs.

Parallel to the case of single equations, several topics can be enumerated as future topics. For instance, in order to construct the super/subsolution methods for time-fractional reaction-diffusion systems, it is essential to develop the comparison principle or, if possible, even the strong maximum principle similarly to that of traditional ones. Regarding the unique continuation property, it worths considering the possibility of assuming the vanishing of only a part of components.

It is an interesting question to ask what happens when we allow one or several orders in \eqref{eq-ibvp-u0} to be exactly 1. Actually one can readily check from Section \ref{sec-alpha} that we can choose some of the orders as 1 in the part of inverse problem. Although one has to reconsider some estimates in the part of forward problem, we conjecture that the results proved in this article still hold true even though the largest order equals to 1.  

We close this article with mentioning another direction of generalizing the orders of time derivatives to $(1,2)$, i.e., coupled time-fractional wave systems modeling viscoelastic wave propagation (e.g.\! \cite{BDES18,KR22}). In these models, the orders $\al_k$ can be the same, whereas the systems should be strongly coupled up to the second spatial derivatives for elasticity. Regardless of their difficulty, it is important to understand the mechanism between viscosity and elasticity.

\appendix
\section{Technical details}\label{sec-app}

\subsection{Strong maximum principle}

We consider a strong maximum principle for the following weakly coupled elliptic system
\begin{equation}\label{equ-w1}
\begin{cases}
(\cA -\bm C)\bm w=\bm F & \mbox{in }\Om,\\
\bm w=\bm0 & \mbox{on }\pa\Om,
\end{cases}
\end{equation}
where $\cA$ is the elliptic operator defined in Section \ref{sect-intro}, $\bm C=(c_{k\ell})_{1\le k,\ell\le K}\in (L^\infty(\Om))^{K\times K}$ is a matrix-valued function and $\bm F=(f_1,\dots,f_K)^{\T}$ is a vector-valued function. Then we have the following strong maximum principle. 

\begin{lem}[Strong maximum principle]\label{lem-smp}
Assume that $\bm C$ satisfies the conditions {\rm\eqref{eq-cond-C1}--\eqref{eq-cond-C2}} in Theorem $\ref{thm-ip}$ and $\bm F\ge \bm 0,\not\equiv\bm 0$ in $\Om,$ that is$,$ $f_k\ge 0$ in $\Om$ for $k=1,\dots,K$ and there exists $k_0\in \{1,\dots,K\}$ such that $f_{k_0}\ge 0,\not\equiv 0$ in $\Om$. Then the solution to \eqref{equ-w1} satisfies $\bm w>\bm 0$ in $\Om,$ that is$,$ $w_k>0$ in $\Om$ for all $k=1,\dots,K$.  
\end{lem}

We refer to \cite{MS95,S89} and the references therein for some results on the (strong) maximum principle for the weakly coupled elliptic systems where the settings on the coupling matrix $\bm C$ are different. For the completeness, we apply the maximum principle in \cite{FM90} and a strong maximum principle for a single elliptic equation to prove the above lemma. 

\begin{proof}
Such a strong maximum principle was also mentioned in \cite[Remark 1.7]{FM90}, but the statement is not clear because the assumption on the coupling coefficients was not correctly written. Thus, here we aim at giving a concise proof by using the maximum principle \cite[Theorem 1.1]{FM90} and a strong maximum principle (e.g.\! \cite[Theorem 4, Chapter 6]{E98}, \cite[Theorem 8.19]{GT01}). 

Since we assume the coupling matrix $\bm C$ satisfies the conditions \eqref{eq-cond-C1}--\eqref{eq-cond-C2} in Theorem \ref{thm-ip} and $\bm F\ge \bm 0$ in $\Om$, we can apply Theorem 1.1 in \cite{FM90} to obtain $\bm w\ge \bm 0$ in $\Om$, that is, $w_k\ge 0$ in $\Om$ for $k=1,\dots,K$. Moreover, the assumption $\bm F\ge\bm 0,\not\equiv\bm 0$ implies that there exists an index $k_0\in\{1,\dots,K\}$ such that $f_{k_0}\ge 0,\not\equiv 0$ in $\Om$. Then we pick up the $k_0$-th equation in \eqref{equ-w1}, which reads
$$
(\cA_{k_0}-c_{k_0k_0})w_{k_0}=f_{k_0}+\sum_{\substack{\ell=1,\dots,K\\
\ell\ne k_0}}c_{k_0\ell}w_\ell=:g.
$$
According to the conditions \eqref{eq-cond-C1}--\eqref{eq-cond-C2}, we have $c_{k_0k_0}\le0$ and $c_{k_0\ell}\ge0$ for $\ell\ne k_0$, and hence $g\ge 0,\not\equiv 0$. Therefore, we can employ a strong maximum principle for a single elliptic equation to conclude that $w_{k_0}>0$ in $\Om$. It remains to prove the other components of $\bm w$ are also positive. To this end, we rewrite the equations \eqref{equ-w1} in each component as follows
\begin{equation}\label{eq-comp}
(\cA_k-c_{k k})w_k=f_k+\sum_{\substack{\ell=1,\dots,K\\
\ell\ne k}}c_{k\ell}w_\ell=:g_k,\quad k=1,\dots,K.
\end{equation}
For each $k\ne k_0$, by the condition \eqref{eq-cond-C1} and $\bm w\ge \bm 0$, we have 
$$
f_k+\sum_{\substack{\ell=1,\dots,K\\
\ell\ne k,k_0}}c_{k\ell}w_\ell\ge0. 
$$
On the other hand, the condition \eqref{eq-cond-C1} reads $c_{k k_0}\ge 0,\not\equiv 0$, which, together with $w_{k_0}>0$ in $\Om$, implies $c_{k k_0}w_{k_0}\ge 0,\not\equiv 0$. Therefore, we obtain $g_k\ge 0,\not\equiv 0$ for any $k\in \{1,\dots,K\}\setminus\{k_0\}$. This completes the proof by applying the strong maximum principle (e.g.\! \cite[Theorem 4, Chapter 6]{E98}) for \eqref{eq-comp}.
\end{proof}

\begin{rem}
{\rm According to the proof, we readily find that one can prove another strong maximum principle by releasing the condition \eqref{eq-cond-C1} (to include the decoupled case) and assume stronger assumption on $\bm F$. That is, we have}
\end{rem}

\begin{lem}[Another strong maximum principle]\label{lem-asmp}
Assume that $\bm C$ satisfies the conditions \eqref{eq-cond-C1'} in Remark $\ref{rem-ip2},$ \eqref{eq-cond-C2} in Theorem $\ref{thm-ip}$ and $f_k\ge 0,\not\equiv 0$ in $\Om$ for all $k\in\{1,\dots,K\}$. Then the solution to \eqref{equ-w1} satisfies $\bm w>\bm 0$ in $\Om,$ that is$,$ $w_k>0$ in $\Om$ for all $k=1,\dots,K$.  
\end{lem}

\subsection{Proof of \eqref{eq-g}}

We will prove $g(s)$ is bounded on $[0,1]$ and satisfies $\lim_{s\to0}g(s)=0$.

\begin{proof}
Recall that
$$
g(s)=s^{1-\al_K}\int_0^\infty t^{-\al_K}f(t)\,\e^{-s t}\,\rd t.
$$
By direct calculation, it is not difficult to prove the boundedness of the above defined function $g(s)$ on $[0,1]$. It remains to show $\lim_{s\to0}g(s)=0$. For this, in view of the fact that $\lim_{t\to+\infty}f(t)=0$, it follows that for any $\ve>0$, there exists $N=N(\ve)>0$ such that $f(t)\le\ve$ for any $t>N$. Therefore, we break the integral in the definition of $g(s)$ into two parts: $[0,N]$ and $(N,\infty)$, we arrive at 
$$
g(s)=s^{1-\al_K}\int_0^N t^{-\al_K}f(t)\,\e^{-s t}\,\rd t+s^{1-\al_K}\int_N^\infty t^{-\al_K}f(t)\,\e^{-s t}\,\rd t=:I_1+I_2.
$$
We give estimates for $I_1$ and $I_2$ separately. For $I_2$, we see that
$$
I_2\le\ve\,s^{1-\al_K}\!\!\int_N^\infty t^{-\al_K}\e^{-s t}\,\rd t\le\ve\,s^{1-\al_K}\!\!\int_0^\infty t^{-\al_K}\e^{-s t}\,\rd t\le\ve\,s^{1-\al_K}s^{\al_K-1}=\ve.
$$
For $I_1$, we choose sufficiently small $\de>0$ so that $\de<\f1N \ve^{\f1{1-\al_K}}$, and we conclude from the boundedness of the function $f(t)$ that
$$
I_1\le C\,s^{1-\al_K}\int_0^N t^{-\al_K}\,\e^{-s t}\,\rd t\le C\,s^{1-\al_K}\f{N^{1-\al_K}}{1-\al_K}
\le\f{C\,\ve}{1-\al_K},\quad \forall\,s<\de.
$$
Collecting all the above estimates, we finally obtain for any $\ve>0$, there exists $\de<\f1N\ve^{\f1{1-\al_K}}$ such that
$$
g(s)\le C\,\ve,\quad\forall\,s<\de,
$$
that is $\lim_{s\to0}g(s)=0$. We finish the proof of \eqref{eq-g}.
\end{proof}

\section*{Acknowledgment}

The authors thank the anonymous referees for their valuable comments.
Z.\! Li is supported by the National Natural Science Foundation of China (NSFC nos.\! 12271277, 11801326).
X.\! Huang is supported by Grant-in-Aid for JSPS Fellows 20F20319 and a JSPS Postdoctoral Fellowship for Overseas Researchers, Japan Society for the Promotion of Science (JSPS).
Y.\! Liu is supported by Grant-in-Aid for Early Career Scientists 20K14355 and 22K13954, JSPS.


\end{document}